\documentclass{amsart}
\usepackage{newcent}

\usepackage{mathrsfs}
\usepackage{amsthm}
\usepackage{amstext}
\usepackage{amssymb}
\usepackage{hyperref}

\usepackage[%backend=biber,
backend=bibtex,
style=numeric,sortcites,
maxbibnames=99
]{biblatex}
\usepackage{enumitem}

\numberwithin{equation}{section}
\numberwithin{figure}{section}
\theoremstyle{plain}
\newtheorem{theorem}{\protect\theoremname}[section]
  \theoremstyle{plain}
  \newtheorem{lemma}[theorem]{\protect\lemmaname}
  \theoremstyle{plain}
  \newtheorem{proposition}[theorem]{\protect\propositionname}
  \theoremstyle{plain}
  \newtheorem{corollary}[theorem]{\protect\corollaryname}
  \theoremstyle{plain}

\theoremstyle{remark}
\newtheorem{remark}[theorem]{Remark}

\theoremstyle{definition}
    \newtheorem{definition}[theorem]{\protect\definitionname}

\providecommand{\corollaryname}{Corollary}
\providecommand{\lemmaname}{Lemma}
\providecommand{\propositionname}{Proposition}
\providecommand{\theoremname}{Theorem}
\providecommand{\conjecturename}{Conjecture}
\providecommand{\definitionname}{Definition}

\newcommand{\A}{\mathcal{A}}
\newcommand{\C}{\mathcal{C}}

\newcommand{\Z}{\mathbb{Z}}
\newcommand{\T}{\mathbb{T}}

\newcommand{\N}{\mathbb{N}}
\newcommand{\B}{\mathcal{B}}
\newcommand{\E}{\mathcal{E}}
\newcommand{\go}{G^{(0)}}
\newcommand\ho{H^{(0)}}
\newcommand{\V}{\mathcal{V}}
\newcommand{\K}{\mathcal{K}}
\newcommand{\Prim}{\operatorname{Prim}}

\newcommand{\supp}{\operatorname{supp}}
\newcommand\red{\text{\upshape{red}}}
\newcommand{\clsp}{\operatorname{\overline{\lsp}}}
\newcommand{\lsp}{\operatorname{span}}

\newcommand{\Iso}{\operatorname{Iso}}
\newcommand\set[1]{\{\,#1\,\}}
\usepackage[normalem]{ulem} % sout stuff1
\usepackage{color}
\definecolor{Dgreen}{cmyk}{0.93,0.33,0.92,0.25} %% Dartmouth Green!

\begin{document}

\title{A Stabilization Theorem for Fell Bundles over groupoids}
\author{Marius Ionescu}

\address{Department of Mathematics\\  United States Naval Academy\\
  Annapolis, MD 21402 USA}

\email{ionescu@usna.edu}

\author{Alex Kumjian}

\address{Department of Mathematics \\ University of Nevada\\ Reno NV 89557 USA}

\email{alex@unr.edu}

\author{Aidan Sims}

\address{School of Mathematics and Applied Statistics\\ University of
  Wollongong\\ NSW 2522, Australia}

\email{asims@uow.edu.au}

\author{Dana P. Williams}

\address{Department of Mathematics\\ Dartmouth College \\ Hanover, NH
03755-3551 USA}

\email{dana.williams@Dartmouth.edu}

\thanks{The first, second, and fourth  authors were partially
supported by their individual grants from the  Simons
Foundation. \\
The first author was also partially supported  by a Junior
NARC grant from the United States Naval Academy.\\
The third author was supported by Australian Research Council grant DP150101595.\\
The second author would like to thank the first and third authors for their hospitality
and support. We thank the anonymous referee for helpful suggestions.}



\begin{abstract}
  We study the $C^*$-algebras associated to upper-semicontinuous Fell
  bundles over second-countable Hausdorff groupoids. Based on ideas
  going back to the Packer--Raeburn ``Stabilization Trick," we
  construct from each such bundle a groupoid dynamical system whose
  associated Fell bundle is equivalent to the original bundle. The
  upshot is that the full and reduced $C^*$-algebras of any saturated
  upper-semicontinuous Fell bundle are stably isomorphic to the full
  and reduced crossed products of an associated dynamical system. We
  apply our results to describe the lattice of ideals of the
  $C^*$-algebra of a continuous Fell bundle by applying Renault's
  results about the ideals of the $C^*$-algebras of groupoid crossed
  products. In particular, we discuss simplicity of the Fell-bundle
  $C^*$-algebra of a bundle over $G$ in terms of an action, described
  by the first and last named authors, of $G$ on the primitive-ideal
  space of the $C^*$-algebra of the part of the bundle sitting over
  the unit space. We finish with some applications to twisted
  $k$-graph algebras, where the components of our results become more
  concrete.
\end{abstract}

\maketitle

\section{Introduction}\label{sec:prelim}

The construction of a $C^*$-algebra from a Fell bundle over a groupoid
subsumes all the usual ways of building a $C^*$-algebra out of group
or groupoid.  Special cases include group and groupoid $C^*$-algebras
(with or without a twisting cocycle), group and groupoid dynamical
systems, $C^{*}$-algebras associated to twists over groupoids, Green's
twisted dynamical systems, and even twisted versions of groupoid
dynamical systems, just to name the most common.

An important task in understanding the structure of any $C^*$-algebra
is understanding its ideals. Quite a bit is known about the ideal
structure of group and groupoid $C^*$-algebras, and also of
crossed-products associated to various types of $C^*$-dynamical
systems (cf., \cite{effhah:mams67, sau:ma77,sau:jfa79,gooros:im79,
  ren:jot91, ionwil:iumj09}). So it would be very useful to be able to
``lift'' results about the ideal structure of groupoid crossed
products to results about Fell-bundle $C^*$-algebras. A result to this
effect for continuous Fell bundles over \'etale groupoids was
established by the second author in \cite{dkr:ms08}: for such Fell
bundles, the reduced $C^*$-algebra is Morita equivalent to the reduced
$C^*$-algebra of a groupoid dynamical system. Muhly sketched in
\cite{muh:cm01} how one can extend Kumjian's techniques to non-\'etale
groupoids, but still only for continuous bundles. Here we elaborate on
a variation of Muhly's approach and in doing so prove a generalization
of his result (Theorem~\ref{thm:main}): given a second-countable
saturated upper-semicontinuous Fell bundle $p:\B\to G$ over a
second-countable Hausdorff locally compact groupoid we show that there
is a groupoid dynamical system (in the sense of \cite{muhwil:nyjm08})
such that the Fell bundle associated with the groupoid dynamical
system is equivalent with $\B$ (see \cite{muhwil:dm08}). Therefore the
full and reduced $C^*$-algebras of $\B$ are Morita equivalent to the
full and reduced $C^*$-algebras, respectively, of the corresponding
groupoid dynamical system. Our techniques are a generalization of
Kumjian and Muhly's arguments and special care is needed due to the
upper-semicontinuity assumption. % We note that Buss, Meyer, and Zhu
% \cite{bmz:pems13} proved, using different methods, the same result for
% (necessarily continuous) Fell bundles over groups.
The stabilization theorem can also be derived from the results in \cite{bmz:pems13} since
a Fell bundle over a groupoid is an example of what the authors call a weak
action.\footnote{The stabilbization theorem fails for Fell bundles over non-Hausdorff
groupoids (see \cite[\S7]{busmey:jng16}).}

Our stabilization theorem is inspired by the Packer--Raeburn
``Stabilization Trick" (see \cite[Theorem~3.4]{pacrae:mpcps89}) for
twisted actions of groups, which in turn builds on Quigg's version of
Takai duality for such actions (see \cite[Theorem~3.1]{qui:iumj86}).

As an application of our results, we apply the powerful results that
Renault proved in \cite{ren:jot91} about the lattice of ideals of the
$C^*$-algebra of a groupoid dynamical system to characterize the ideal
lattice and the simplicity of the $C^*$-algebra of a continuous Fell
bundle under certain amenability assumptions (Corollaries
\ref{cor:lat}~and~\ref{cor:sim}). As a second application, we describe
a collection of examples of Fell bundles arising from twisted
$k$-graph $C^*$-algebras, and apply our main result to provide a
method for calculating their primitive-ideal spaces in some special
cases.

\section{Background about Fell bundles}

We recall next some of the definitions and facts that are needed to
prove our results. In this note we assume that $G$ is a
second-countable Hausdorff locally compact groupoid endowed with a
Haar system $\{ \lambda^x\}_{x\in \go}$ \cite{ren:groupoid}. We write
$r:G\to\go$ for the \emph{range} map $r(g)=gg^{-1}$ and $s:G\to\go$
for the \emph{source} map $s(g)=g^{-1}g$. Recall that
$\supp \lambda^x=G^x:=r^{-1}(x)$ for all $x\in \go$. For $x\in\go$ we
let $\lambda_x$ be the image of $\lambda^x$ under
inversion. Therefore, the support of $\lambda_x$ is
$G_x:=s^{-1}(x)$. For two locally compact groupoids $G$ and $H$, a
$G$--$H$ equivalence \cite[Definition 2.1]{MuReWi_JOT87} is a locally
compact Hausdorff space $Z$ which is simultaneously a free and proper
left $G$-space and a free and proper right $H$-space such that the $G$
and $H$ actions commute and such that the moment maps $r:Z\to\go$ and
$s:Z\to \ho$ induce homeomorphisms of $Z/H$ onto $\go$ and
$G\backslash Z$ onto $\ho$, respectively.  Consequently, given
$z,w\in Z$ such that $s(z)=s(w)$ there is a unique ${_{G}[z,w]}\in G$
such that ${_G[z,w]}\cdot w=z$.  Similarly, if $r(z)=r(w)$, then there
is a unique $[z,w]_{H}\in H$ such that $z\cdot [z,w]_{H}=w$.  If $G$
is a locally compact groupoid then $G$ is trivially a $G$--$G$
equivalence.

Following \cite{muhwil:dm08} (see also \cite{fd:representations1,
  fd:representations2,kum:pams98}), a Fell bundle over $G$ is an
upper-semicon\-tinuous Banach bundle $p:\B\to G$ endowed with a
partially defined multiplication and an involution that respect $p$
such that the fibres $A(x)$ over $x\in\go$ are $C^*$-algebras and each
fibre $B(g)$ is an $A(r(g))$--$A(s(g))$ imprimitivity bimodule with
respect to the inner products and actions induced by the
multiplication in $\B$. Note that for $x\in \go$ we will write both
$A(x)$ and $B(x)$ for the fiber over $x$ to distinguish its dual role
as a $C^*$-algebra and as the trivial imprimitivity bimodule over
$A(x)$. The $C^*$-algebra $A:=\Gamma_0(\go;\B)$ is called the
$C^*$-algebra over the unit space. We assume that our Fell bundles are
second countable and saturated and they have enough sections (in the
sense that evaluation of continuous sections at $g$ is surjective onto
$B(g)$ for each $g$ \cite[p.~16]{muhwil:nyjm08}). Since $B(g)$ is an
$A(r(g))$--$A(s(g))$ imprimitivity bimodule, the Rieffel
correspondence \cite{rie:pspm82b,rie:aim74,rw:morita} induces a
homeomorphism $h_g :\Prim A(s(g))\to \Prim A(r(g))$. By
\cite[Proposition~2.2]{ionwil:hjm11} there is a continuous action of
$G$ on $\Prim A$ such that
\begin{equation}
  \label{eq:action}
  g\cdot (s(g),P)=(r(g),h_g(P)).
\end{equation}

Suppose that $G$ and $H$ are locally compact Hausdorff groupoids, that
$p_G:\B\to G$ is a Fell bundle over $G$ and that $p_H:\C\to H$ is a
Fell bundle over $H$. A $\B$--$\C$ equivalence \cite[Definition
6.1]{muhwil:dm08} consists of a $G$--$H$ equivalence $Z$ and an
upper-semicontinuous bundle $q:\E\to Z$ such that $\B$ acts on the
left of $\E$, $\C$ acts on the right on $\E$, the two actions commute,
and there are sesquilinear maps $_\B\langle\cdot,\cdot\rangle$ from
$\E *_{s}\E$ to $\B$ and $\langle\cdot,\cdot\rangle_\C$ from
$\E *_{r}\E$ to $\C$ such that
\begin{enumerate}[label= (\alph*)]
\item $p_G(_\B\langle e,f\rangle)={_G[q(e),q(f)]}$ and
  $p_H(\langle e,f\rangle_\C)=[q(e),q(f)]_H$,
\item $_\B\langle e,f\rangle^*= \;_\B\langle f,e\rangle$ and
  $\langle e,f\rangle_\C^*=\langle f,e\rangle_\C$,
\item $_\B\langle b\cdot e,f\rangle=b\cdot \;_\B\langle e,f\rangle$
  and $\langle e,f\cdot c\rangle_\C=\langle e,f\rangle_\C \cdot c$,
\item $_\B\langle e,f\rangle \cdot g=e\cdot \langle f,g\rangle_\C$,
\end{enumerate}
and such that, with the $\B$ and $\C$ actions and with the inner
products coming from $_\B\langle\cdot, \cdot \rangle$ and
$\langle\cdot, \cdot\rangle_\C$, each $E(z)$ is a
$B(r(z))$--$C(s(z))$-imprimitivity bimodule.  If $\B$ and $\C$ are
equivalent Fell bundles then $C^*(G;\B)$ and $C^*(H;\C)$ are Morita
equivalent \cite[Theorem~6.4]{muhwil:dm08} and so are
$C^*_{\red}(G;\B)$ and $C^*_{\red}(H;\C)$
\cite{SimWil_NYJM13,kum:pams98}.

An important example for this note is the Fell bundle associated with
a groupoid dynamical system. Let $\pi:\A\to\go$ be an
upper-semicontinuous $C^*$-bundle over $\go$.  Assume that
$(\A,G,\alpha)$ is a groupoid dynamical system
\cite{ren:jot87,muhwil:nyjm08}. Then one can define a Fell bundle that
we denote by $\sigma:\A\rtimes_\alpha G\to G$ as in \cite[Example
2.1]{muhwil:dm08}:
\begin{equation}
  \label{eq:Fellbdldynsyst}
  \A\rtimes_\alpha G:=r^*\A=\{\,(a,g)\in \A\times G\,:\,\pi(a)=r(g)\,\}
\end{equation}
with multiplication $(a,g)(b,h)=(a\alpha_g(b),gh)$ whenever
$s(g)=r(h)$, and the involution given by
$(a,g)^*=(\alpha_g^{-1}(a^*),g^{-1})$.

If $V$ is a right Hilbert module over a $C^*$-algebra $A$
\cite{lan:hilbert,rw:morita,rie:pspm82b,kum:pams98}, then there is a left $A$-module
$V^*$ with a conjugate linear isomorphism from $V$ to $V^*$, written $v\mapsto v^*$ such
that $av^*=(va^*)^*$ and $_A\langle u^*,v^*\rangle=\langle u,v\rangle_A$. Rank-one
operators on $V$ are defined via $\theta_{u,v}(w)=u\cdot \langle v, w\rangle $. The set
of compact operators $\K(V)$ on $V$ is the closure of the linear span of rank-one
operators.  Compact operators are adjointable and $\K(V)$ is a $C^*$-algebra with respect
to the operator norm \cite[pp.~9--10]{lan:hilbert}. If $V$ is full then $V$ is a
$\K(V)$--$A$ imprimitivity bimodule where the left $\K(V)$-inner product is
$_{\K(V)}\langle u,v\rangle=\theta_{u,v}$. Hence $\K(V)$ and $A$ are Morita equivalent.
The map $u\otimes v^*\mapsto \theta_{u,v}$ extends to an isomorphism $V\otimes_A V^*
\cong \K(V)$ (we make this identification in the sequel without comment).

\section{Main result and applications}\label{sec:main}

Let $G$ be a locally compact Hausdorff groupoid endowed with a Haar
system $\{\lambda_x\}_{x\in\go}$. Let $p:\B \to G$ be a
second-countable saturated Fell bundle over $G$ and let
$A=\Gamma_0(\go;\B)$ be the $C^*$-algebra over $\go$.  We construct an
upper-semicontinuous $C^*$-bundle $k:\K(V)\to\go$ and an action
$\alpha$ of $G$ on $\K(V)$ such that $(\K(V),G,\alpha)$ is a groupoid
dynamical system and such that $\B$ and $\K(V)\rtimes_\alpha G$ are
equivalent Fell bundles, generalizing the construction and results of
\cite[Section~4]{kum:pams98} and \cite[Section~4]{muh:cm01}. We break
our construction into a series of lemmas and propositions.

For $x\in\go$ let $V(x)$ be the closure of $\Gamma_c(G_x;\B)$ under
the pre-inner product
\[
\langle
\xi,\eta\rangle_{A(x)}=\int_{G_x}\xi(\gamma)^*\eta(\gamma)d\lambda_x(\gamma).
\]
Then $V(x)$ is a full right Hilbert $A(x)$-module with the right
action given by $(\xi\cdot a)(\gamma)=\xi(\gamma)a$ for
$\xi\in \Gamma_c(G_x;\B)$ and $a\in A(x)$
\cite[Lemma~2.16]{rw:morita}. Let $V:=\bigsqcup_{x\in \go} V(x)$ and
let $\nu:V\to \go$ be the projection map.

\begin{lemma}\label{lemm1}
  With notation as above, the map
  $x\mapsto \langle \xi,\eta\rangle_{A(x)}$ is continuous for all
  $\xi,\eta\in \Gamma_c(G;\B)$.  Moreover, there is a unique topology
  on $V$ such that $\nu:V\to\go$ is an upper-semicontinuous Banach
  bundle over $\go$ of which $x\mapsto \hat f(x)=f|_{G_{x}}$ is a
  continuous section for each $f\in \Gamma_{c}(G;\B)$. The space
  $\V:=\Gamma_0(\go;V)$ is then a full Hilbert $A$-module.
\end{lemma}
\begin{proof}
  The first assertion follows from the fact that
  $f^*g \in \Gamma_c(G;\B)$ for all $f, g \in \Gamma_c(G;\B)$ (see
  \cite[Corollary 1.4]{muhwil:dm08}).  Given a section
  $f\in \Gamma_c(G;\B)$ define a section $\hat{f}$ of $V$ by
  $\hat{f}(x)=f|_{G_x}$. As
  \[
  \Vert \hat{f}(x)\Vert =\Vert \langle
  \hat{f},\hat{f}\rangle_{A(x)}\Vert^\frac{1}{2},
  \]
  the map $x\mapsto \Vert \hat{f}(x)\Vert$ is the composition of a
  continuous function and an upper-semicontinuous function, and hence
  itself upper semicontinuous. The vector-valued Tietze Extension
  Theorem~\cite[Proposition~A.5]{muhwil:dm08} implies that the set
  $\{\,\hat{f}(x)\,:\,f\in \Gamma_c(G;\B)\,\}$ is dense in $V(x)$. The
  Hofmann--Fell theorem (see, for example,
  \cite[Theorem~II.13.18]{fd:representations1}, \cite{hof:74},
  \cite{hof:lnim77}, and \cite[Theorem~1.2]{ionwil:hjm11}) implies
  that there is a unique topology such that $V$ is an upper
  semicontinuous Banach bundle over $\go$ and such that
  $\Gamma(\go;V)$ contains $\{\,\hat{f}$\,:\,
  $f\in \Gamma_c(G;\B)\,\}$.  It is now easy to see that $\V$,
  equipped with the natural inner product and right $A$-action, is a
  Hilbert $A$-module.
\end{proof}

For the next lemma let $\K(V) := \bigsqcup_{x\in\go} \K(V(x))$ and let
$k : \K(V) \to \go$ be the canonical map.

\begin{lemma}\label{lemm2}
  Resume the notation of Lemma~\ref{lemm1}. Then $\K(V)$ has a unique
  topology such that $k:\K(V)\to \go$ is
  an upper semicontinuous $C^*$-bundle admitting a $C_0(\go)$-linear
  isomorphism of $\K(\V)$ onto $\Gamma_0(\go;\K(V))$, and each
  $\K(V)(x) \cong \K(V(x))$.
\end{lemma}
\begin{proof}
  Since $A=\Gamma_0(\go;\B)$ is a $C_0(\go)$-algebra there is a
  continuous map $\sigma_A:\Prim A\to \go$ (see, for example,
  \cite[Theorem~C.26]{wil:crossed}) given by $\sigma(I) = x$ if and
  only if $I$ contains the ideal generated by functions in $C_0(\go)$
  that vanish at $x$. Since $\V$ is an $\K(\V)$--$A$ imprimitivity
  bimodule, the Rieffel correspondence restricts to a homeomorphism
  $h:\Prim \K(\V) \to \Prim A$ (see, for example,
  \cite[Proposition~3.3.3]{rw:morita}). Therefore one obtains by
  composition a continuous map $h\circ \sigma_A :\Prim
  \K(\V)\to\go$.
  Theorem~C.26 of \cite{wil:crossed} implies that $\K(\V)$ is a
  $C_0(\go)$-algebra and that $k:\K(V)\to \go$ is an
  upper-semicontinuous $C^*$-bundle such that there is a
  $C_0(\go)$-linear isomorphism of $\K(\V)$ onto
  $\Gamma_0(\go,\K(V))$. The fiber of $\K(\V)$ over $x\in\go$ is
  isomorphic to $\K(V(x))$ by construction, so $\K(V)(x)$ is
  isomorphic to $\K(V(x))$ for all $x\in\go$.
\end{proof}

The following lemma will be useful in the definition of the action of
$G$ on $\K(V)$ and the equivalence between Fell bundles.
\begin{lemma}\label{lem:isom_HM}
  For $g\in G$, the map
  $\beta_g:\Gamma_c(G_{r(g)};\B)\odot B(g)\to \Gamma_c(G_{s(g)};\B)$
  defined on elementary tensors by
  \[
  \beta_g(\xi\otimes b)(\gamma)= \xi(\gamma g^{-1})b
  \]
  extends to an isometric isomorphism
  $\beta_g:V(r(g))\otimes_{A(r(g))} B(g)\to V(s(g))$ of Hilbert
  $A(s(g))$-modules.
\end{lemma}
\begin{proof}
  Fix $g \in G$. Let $\xi, \eta \in \Gamma_c(G_{r(g)};\B)$ and
  $a,b \in B(g)$. Using left invariance at the third equality, we
  calculate:
  \begin{align*}
    \langle \beta_g(\xi\otimes a), \beta_g(\eta \otimes b)\rangle
    &= \int_{G_{s(g)}} \beta_g(\xi\otimes a)^*(\gamma)\beta_g(\eta\otimes b)
      (\gamma)d\lambda_{s(g)}(\gamma)\\
    &= \int_{G_{s(g)}} a^* \xi^*(\gamma g^{-1})\eta(\gamma g^{-1})b
      d\lambda_{s(g)}(\gamma)\\
    &=a^* \int_{G_{r(g)}}\xi^*(\gamma)\eta(\gamma) d\lambda_{r(g)}(\gamma) b\\
    &=a^*\langle \xi,\eta\rangle_{A(r(g))} b\\
    &=\langle \xi\otimes a, \eta\otimes b\rangle_{A(s(g))}
  \end{align*}
  Thus $\beta_g$ preserves the inner-product, which implies first that
  it is isometric, and second---by right-linearity of the inner
  product---that it preserves the right $A_{s(g)}$-action.

  To check that the range of $\beta_g$ is dense in $V(s(g))$ fix
  $\xi^\prime \in \Gamma_c(G_{s(g)};\B)$.  There is an approximate
  identity $\{ c_\nu \}_\nu$ of $A(s(g))$ of the form
  $c_\nu = \sum_{i=1}^{n_\nu} b_{\nu,i}^*b_{\nu,i}$ where
  $b_{\nu,i} \in B(g)$ for all $\nu, i$.  For each $\nu, i$, define
  $\xi_{\nu,i}\in \Gamma_c(G_{r(g)};\B)$ by
  $\xi_{\nu,i}(\gamma)=\xi^\prime(\gamma g)b_{\nu,i}^*$.  Then the net
  $\big\{\sum_{i=1}^{n_\nu} \beta_g(\xi_{\nu,i}\otimes
  b_{\nu,i})\big\}_\nu$
  converges to $\xi^\prime$.  Hence $\beta_g$ extends to an isometric
  isomorphism of Hilbert $A(s(g))$-modules.
\end{proof}
As an easy consequence of Lemma~\ref{lem:isom_HM} we obtain that, for
$g\in G$, the map
$\beta_g^*: B(g)\odot \Gamma_c(G_{s(g)};\B)^*\to
\Gamma_c(G_{r(g)};\B)^*$ defined on elementary tensors by
\[
\beta_g^*(b\otimes \eta^*)=\bigl(\beta_{g^{-1}}(\eta\otimes
b^*)\bigr)^*
\]
extends to an isometric isomorphism
$B(g)\otimes V(s(g))^* \cong V(r(g))^*$ of left Hilbert
$A(r(g))$-modules.  It is important in the following to keep in mind
that the maps $\beta_g$ and $\beta_g^*$ are onto. Moreover,
$\clsp\{\,\beta_g(\xi\otimes b)\,:\xi\in V(r(g)),b\in B(g)\,\} =
V(s(g))$
and
$\clsp\{\,\beta_g^*(b\otimes \eta^*)\,:\,b\in B(g),\eta\in
V(s(g))\,\}=V(r(g))^*$.
Therefore, for $g\in G$, the map $\alpha_g$ defined by
\begin{equation}
  \label{eq:1}
  \alpha_g(\beta_g(\xi\otimes b)\otimes \eta^*)=\xi\otimes
  \beta^*_g(b\otimes \eta^*),
\end{equation}
for $\xi\in V(r(g)),b\in B(g)$, and $\eta\in V(s(g))$, extends to an
isomorphism between $\K(V(s(g)))$ and $\K(V(r(g)))$. We collect in the
following lemma a few useful facts about the maps $\beta_g$ and
$\alpha_g$.
\begin{lemma}
  Resume the notation of Lemma~\ref{lem:isom_HM} and let
  $\{\alpha_g : g \in G\}$ be the isomorphisms defined in
  \eqref{eq:1}.
  \begin{enumerate}[label={\upshape(\alph*)}]
  \item For $g \in G$ we have
    \begin{equation}
      \label{eq:2}
      \beta_{g^{-1}}(\beta_g(\xi\otimes b_1)\otimes b_2^*)=\xi\cdot
      {}_{A(r(g))}\langle b_1,b_2\rangle=\xi\cdot (b_1b_2^*).
    \end{equation}
    for all $\xi\in V(r(g))$ and $b_1,b_2\in B(g)$.
  \item For composable $g,h\in G$, we have
    \begin{equation}
      \label{eq:7}
      \beta_{gh}(\xi\otimes b_1b_2)=\beta_h(\beta_g(\xi\otimes
      b_1)\otimes b_2).
    \end{equation}
    for all $\xi \in V(r(g))=V(r(gh))$, $b_1\in B(g)$ and
    $b_2\in B(h)$.
  \item For $g\in G$ and $T\in \K(V(s(g)))$, we have
    \begin{equation}
      \label{eq:8}
      T\circ \beta_g=\beta_g\circ \bigl(\alpha_g(T)\otimes \operatorname{Id}\bigr).
    \end{equation}
  \end{enumerate}
\end{lemma}
\begin{proof}
  (a) Let $g\in G$, $\xi\in \Gamma_c(G_{r(g)};\B)$, $b_1,b_2\in B(g)$,
  and $\gamma\in G_{r(g)}$. Then
  \[
  \beta_{g^{-1}}(\beta_g(\xi\otimes b_1)\otimes
  b_2^*)(\gamma)=\beta_g(\xi\otimes b_1)(\gamma
  g)b_2^*=\xi(\gamma)b_1b_2^*.
  \]
  The result follows.

  (b) Let $\xi\in \Gamma_{c}(G_{r(g)};\B)=\Gamma_{c}(G_{r(gh)};\B)$
  and let $\gamma\in G_{s(h)}=G_{s(gh)}$. Then
  \[
  \beta_{gh}(\xi\otimes b_1b_2)(\gamma)=\xi(\gamma
  h^{-1}g^{-1})b_1b_2=\beta_g(\xi\otimes b_1)(\gamma
  h^{-1})b_2=\beta_h(\beta_g(\xi\otimes b_1)\otimes b_2)(\gamma).
  \]

  (c) Assume that $T$ is a rank-one operator $T=u\otimes v^*$ with
  $u,v\in \Gamma_c(G_{s(g)};\B)$ and assume that
  $u=\beta_g(u^\prime\otimes b^\prime)$ for some
  $u^\prime\otimes b^\prime\in V(r(g))\otimes B(g)$. Let
  $\xi\in V(r(g))$ and $b\in B(g)$. Then
  \begin{align*}
    T(\beta_g(\xi\otimes b))&=u\cdot \langle v,\beta_g(\xi\otimes
                              b)\rangle_{A(s(g))}\\
                            &=u \int_{G_{s(g)}}v^*(\gamma)\beta_g(\xi\otimes
                              b)(\gamma)d\lambda_{s(g)}(\gamma)\\
                            &=u \int_{G_{s(g)}}v^*(\gamma)\xi(\gamma
                              g^{-1})b\,d\lambda_{s(g)}(\gamma)\\
                            &=u \int_{G_{s(g)}}v^*(\gamma)\xi(\gamma
                              g^{-1})\,d\lambda_{s(g)}(\gamma)b.
  \end{align*}
  Also,
  \begin{align*}
    \beta_g(\alpha_g(T)\xi\otimes
    b)&=\beta_g(\alpha_g(\beta_g(u^\prime \otimes b^\prime)\otimes
        v^*)\xi \otimes b)\\
      &=\beta_g((u^\prime \otimes \beta_g^*(b^\prime \otimes
        v^*)\xi)\otimes b)\\
      &=\beta_g(u^\prime\langle \beta_{g^{-1}}(v\otimes b^{\prime
        *}),\xi\rangle_{A(r(g))}\otimes b)\\
      &=\beta_g(u^\prime\otimes b^\prime)\int_{G_{r(g)}}v^*(\gamma
        g)\xi(\gamma)\,d\lambda_{r(g)}(\gamma)b,
  \end{align*}
  where the last equality follows from an easy computation. Using the
  invariance of the Haar system we deduce~\eqref{eq:8} for rank-one
  operators. The result then follows from linearity and continuity.
\end{proof}

\begin{proposition}
  With $\K(V)$ defined in Lemma~\ref{lemm2} and with $\alpha$ defined
  in \eqref{eq:1}, $(\K(V),G,\alpha)$ is a groupoid dynamical system.
\end{proposition}
\begin{proof}
  We already know that $\alpha_g:\K(V(s(g)))\to \K(V(r(g)))$ is an
  isomorphism. It is not obvious, however, that $\alpha_g$ is a
  *-homomorphism and we provide a proof next. Recall that using the
  identification
  $\K(V(s(g))\simeq V(s(g))\otimes_{A(s(g))} V(s(g))^*$, the product
  of $u_1\otimes v_1^*$ and
  $u_2\otimes v_2^*\in V(s(g))\otimes V(s(g))^*$ is given by
  $u_1\cdot \langle v_1,u_2\rangle_{A(s(g))}\otimes v_2^*$. Let
  $\xi_1,\xi_2\in \Gamma_c(G_{r(g)};\B)$, $b_1,b_2\in B(g)$, and
  $\eta_1,\eta_2\in \Gamma_c(G_{s(g)};\B)$.  Then
  \begin{align*}
    \alpha_g\bigl((\beta_g(\xi_1\otimes b_1)
    &\otimes \eta_1^*)\cdot
      (\beta_g(\xi_2\otimes b_2)\otimes \eta_2^*)\bigr)\\
    &= \alpha_g\bigl(\beta_g(\xi_1\otimes b_1)\cdot \langle
      \eta_1,\beta_g(\xi_2\otimes b_2)\rangle_{A(s(g))}\otimes
      \eta_2^*\bigr)
      \intertext{which, using that $\beta_g$ is a Hilbert module map,
      is}
    &=\alpha_g\bigl(\beta_g(\xi_1\otimes (b_1\cdot  \langle
      \eta_1,\beta_g(\xi_2\otimes b_2)\rangle_{A(s(g))}))\otimes
      \eta_2^*\bigr).
      \intertext{By the definition of $\alpha_g$, this is}
    &= \xi_1\otimes \beta_g^*\bigl(b_1\cdot \langle
      \eta_1,\beta_g(\xi_2\otimes b_2)\rangle_{A(s(g))}\otimes
      \eta_2^*\bigr)\\
    &=\xi_1\otimes \left(\beta_{g^{-1}}\bigl(\eta_2\otimes (b_1\cdot
      \langle\eta_1,\beta_g(\xi_2\otimes
      b_2)\rangle_{A(s(g))})^*\bigr)\right)^*\\
    &= \xi_1 \otimes \left(\beta_{g^{-1}}\bigl(\eta_2\otimes \langle
      \eta_1,\beta_g(\xi_2\otimes
      b_2)\rangle_{A(s(g))}^*b_1^*\bigr)\right)^*\\
    &=\xi_1\otimes \left(\beta_{g^{-1}}\bigl(\eta_2\otimes \langle
      \beta_g(\xi_2\otimes b_2),\eta_1\rangle_{A(s(g))}b_1^*\bigr)\right)^*.
  \end{align*}
  On the other hand,
  \begin{align*}
    \alpha_g(\beta_g(\xi_1\otimes b_1)
    &\otimes \eta_1^*)\cdot
      \alpha_g(\beta_g(\xi_2\otimes b_2)\otimes \eta_2^*)\\
    &=\bigl(\xi_1\otimes \beta_g^*(b_1\otimes \eta_1^*)\bigr)\cdot
      \bigl(\xi_2\otimes \beta_g^*(b_2\otimes \eta_2^*)\bigr) \\
    &=\bigl(\xi_1\otimes (\beta_{g^{-1}}(\eta_1\otimes b_1^*))^*\bigr)\cdot
      \bigl(\xi_2\otimes \beta_g^*(b_2\otimes \eta_2^*)\bigr) \\
    &=\xi_1\cdot \langle \beta_{g^{-1}}(\eta_1\otimes
      b_1^*),\xi_2\rangle_{A(r(g))}\otimes
      \bigl(\beta_{g^{-1}}(\eta_2\otimes b_2^*)\bigr)^*.
      \intertext{Since the tensor product is balanced over
      $A(r(g))$, this is}
    &=\xi_1\otimes \left(\beta_{g^{-1}}(\eta_2\otimes b_2^*)\cdot
      \langle \beta_{g^{-1}}(\eta_1\otimes
      b_1^*),\xi_2\rangle_{A(r(g))}^*\right)^* \\
    &=\xi_1\otimes \left(\beta_{g^{-1}}(\eta_2\otimes \bigl(b_2^*\cdot
      \langle \xi_2,\beta_{g^{-1}}(\eta_1\otimes
      b_1^*)\rangle_{A(r(g))}\bigr))\right)^*.
  \end{align*}
  Using the invariance of the Haar system we have that
  \begin{align*}
    \langle \beta_g(\xi_2\otimes
    b_2),\eta_1\rangle_{A(s(g))}b_1^*
    &=\int_{G_{s(g)}}\beta_g(\xi_2\otimes
      b_2)^*(\gamma)\eta(\gamma)d\lambda_{s(g)}(\gamma)b_1^*
    \\
    &= b_2^*\int_{G_{s(g)}}\xi_2^*(\gamma
      g^{-1})\eta(\gamma)d\lambda_{s(g)}(\gamma)b_1^*\\
    &=b_2^*\int_{G_{r(g)}}\xi_2^*(\gamma)\eta(\gamma
      g)b_1^*d\lambda_{r(g)}(\gamma)\\
    &=b_2^*\int_{G_{r(g)}}\xi_2^*(\gamma)\beta_{g^{-1}}(\eta\otimes
      b_1^*)d\lambda_{r(g)}(\gamma)\\
    &=b_2^*\cdot \langle \xi_2,\beta_{g^{-1}}(\eta\otimes b_1^*)\rangle_{A(r(g))}.
  \end{align*}
  Therefore
  \begin{multline*}
    \alpha_g\bigl((\beta_g(\xi_1\otimes b_1)\otimes \eta_1^*)\cdot
    (\beta_g(\xi_2\otimes b_2)\otimes \eta_2^*)\bigr)\\
    =\alpha_g(\beta_g(\xi_1\otimes b_1)\otimes \eta_1^*)\cdot
    \alpha_g(\beta_g(\xi_2\otimes b_2)\otimes \eta_2^*).
  \end{multline*}
  Take $\xi_1,\xi_2\in V(r(g))$ and $b_1,b_2\in B(g)$. Recall that the
  adjoint of $u\otimes v^*\in V(s(g))\otimes V(s(g))^*$ equals
  $v\otimes u^*$. Then
  \begin{align*}
    \alpha_g\left(\beta_g(\xi_1\otimes b_1)\otimes
    (\beta_g(\xi_2\otimes b_2))^*\right)^*
    &=\left(\xi_1\otimes
      \beta_g^*(b_1\otimes (\beta_g(\xi_2\otimes b_2))^*\right)^*\\
    &= \left(\xi_1 \otimes \bigl(\beta_{g^{-1}}(\beta_g(\xi_2\otimes
      b_2)\otimes b_1^*)\bigr)^*\right)^*\\
    &=\beta_{g^{-1}}(\beta_g(\xi_2\otimes
      b_2)\otimes b_1^*)\otimes \xi_1^*,\\
    \intertext{which, using \eqref{eq:2}, is}
    &=\xi_2\cdot (b_2b_1^*)\otimes \xi_1^*=\xi_2\otimes
      \bigl(\xi_1\cdot (b_1b_2^*)\bigr)^*\\
    &=\xi_2\otimes \beta_g^*(b_2\otimes \beta_g(\xi_1\otimes b_1)^*)\\
    &=\alpha_g\left(\bigl(\beta_g(\xi_1\otimes b_1)\otimes
      (\beta_g(\xi_2\otimes b_2))^*\bigr)^*\right).
  \end{align*}
  Therefore $\alpha_g$ is a *-isomorphism.

  Take $g,h \in G$ such that $s(g)=r(h)$. To prove that
  $\alpha_{gh}=\alpha_g\circ \alpha_h$ let $\xi\in V(r(g))$,
  $a\in B(g)$, $b\in B(h)$, and $\eta\in V(s(h))$. Then
  \begin{align*}
    (\alpha_g\circ \alpha_h)(\beta_h(\beta_g(\xi\otimes a)\otimes
    b)\otimes \eta^*)&=\alpha_g(\beta_g(\xi\otimes a)\otimes
                       \beta^*_h(b\otimes \eta^*))\\
                     &=\xi\otimes \beta^*_g(a\otimes
                       \beta^*_h(b\otimes \eta^*))\\
                     &=\xi\otimes \beta^*_{gh}(ab\otimes \eta^*) \\
                     &=\alpha_{gh}((\beta_h(\beta_g(\xi\otimes a)\otimes
                       b)\otimes \eta^*).
  \end{align*}

  Finally we prove that the action of $G$ on $\K(V)$ is continuous. We
  are going to use the criterion from
  \cite[Proposition~C.20]{wil:crossed} for the convergence of nets in
  an upper-semicontinuous $C^*$-bundle. It suffices to consider nets
  $\{g_i\}\subset G$ such that $g_i\to g$, $\{\xi_i\}$ indexed by a
  set $I$ such that $\xi_i\in \Gamma_c(G_{r(g_i)};\B)$ and
  $\xi_i\to \xi\in \Gamma_c(G_{r(g)};\B)$, $\{b_i\}$ such that
  $b_i\in B(g_i)$ and $b_i\to b\in B(g)$, and $\{\eta_i\}$ such that
  $\eta_i\in \Gamma_c(G_{s(g_i)};\B)$ and
  $\eta_i \to \eta\in \Gamma_c(G_{s(g)};\B)$. Moreover it suffices to
  consider nets $\{\tilde{\xi}_i\},\{\tilde{b}_i\}$, and
  $\{\tilde{\eta}_i\}$ in $\Gamma_c(G;\B)$ such that
  $\tilde{\xi}_i\to \tilde{\xi}$, $\tilde{b}_i\to b$,
  $\tilde{\eta}_i\to \tilde{\eta}$, $\tilde{\xi}|_{G_{r(g_i)}}=\xi_i$,
  $\tilde{b}_i(g_i)=b_i$, and
  $\tilde{\eta}_i|_{G_{s(g_i)}}=\eta_i$. We will write
  $\tilde{\xi}_i(x):=\tilde{\xi}_i|_{G_{x}}$ and similarly for
  $\tilde{\eta}_i(x)$ with $x\in\go$ in order to keep the notation
  simple. Fix $\varepsilon>0$ and fix an index $j\in I$ such that
  $\Vert \tilde{\xi}_j-\tilde{\xi}\Vert<\varepsilon$,
  $\Vert \tilde{b}_j-\tilde{b}\Vert<\varepsilon$, and
  $\Vert \tilde{\eta}_j-\tilde{\eta}\Vert<\varepsilon$. Then, since
  $\beta_{g_i}^*$ are isometric isomorphisms,
  \begin{multline*}
    \alpha_{g_i}(\beta_{g_i}(\tilde{\xi}_j(r(g_i))\otimes
    \tilde{b}_j(g_i))\otimes
    \tilde{\eta}_j(s(g_i))^*)=\tilde{\xi}_j(r(g_i))
    \otimes\beta_{g_i}^*(\tilde{b}_j(g_i)\otimes
    \tilde{\eta}_j(s(g_i))^*)\\
    \to \tilde{\xi}_j(r(g))\otimes\beta_{g}^*(\tilde{b}_j(g)\otimes
    \tilde{\eta}_j(s(g))^*)=\alpha_{g}(\beta_{g}(\tilde{\xi}_j(r(g))\otimes
    \tilde{b}_j(g))\otimes \tilde{\eta}_j(s(g))^*).
  \end{multline*}
  Therefore, if we set
  \begin{gather*}
    a_i:=\alpha_{g_i}(\beta_{g_i}(\tilde{\xi}_i(r(g_i))\otimes
    \tilde{b}_i(g_i))\otimes \tilde{\eta}_i(s(g_i))^*),\\
    u_i:=\alpha_{g_i}(\beta_{g_i}(\tilde{\xi}_j(r(g_i))\otimes
    \tilde{b}_j(g_i))\otimes \tilde{\eta}_j(s(g_i))^*),\\
    a:=\alpha_{g}(\beta_{g}(\tilde{\xi}(r(g))\otimes
    \tilde{b}(g))\otimes \tilde{\eta}(s(g))^*)\quad\text{and}\\
    u:=\alpha_{g}(\beta_{g}(\tilde{\xi}_j(r(g))\otimes
    \tilde{b}_j(g))\otimes \tilde{\eta}_j(s(g))^*),
  \end{gather*}
  then the following hold: $u_i\to u$, $k(u_i)=g_i=k(a_i)$,
  $\Vert a-u\Vert<\varepsilon$, and $\Vert a_i-u_i\Vert <\varepsilon$
  for large $i$. Therefore \cite[Proposition~C.20]{wil:crossed}
  implies that $a_i\to a$, that is
  \[
  \alpha_{g_i}(\beta_{g_i}(\tilde{\xi}_i(r(g_i))\otimes
  \tilde{b}_i(g_i))\otimes \tilde{\eta}_i(s(g_i))^*) \to
  \alpha_{g}(\beta_{g}(\tilde{\xi}(r(g))\otimes \tilde{b}(g))\otimes
  \tilde{\eta}(s(g))^*).
  \]
  Hence the action of $G$ on $\K(V)$ is continuous and
  $(\K(V),G,\alpha)$ is a groupoid dynamical system.
\end{proof}

\begin{remark}\label{rem:continuous}
  Note that if we replace ``upper semicontinuous'' with ``continuous''
  in the hypothesis of Lemma~\ref{lemm1}, Lemma~\ref{lemm2}, and
  Theorem~\ref{thm:main}, then $\nu:V\to \go$ is a continuous Banach
  bundle, $k:\K(V)\to\go$ is a (continuous) $C^*$-bundle
  \cite[Theorem~II.13.18]{fd:representations1} and, hence,
  $(\K(V),G,\alpha)$ is a continuous groupoid dynamical system (in the
  sense of \cite{ren:jot87,ren:jot91}).
\end{remark}

We let $\sigma:\K(V)\rtimes_\alpha G\to G$ be the semi-direct crossed
product Fell bundle. Recall from \eqref{eq:Fellbdldynsyst} that
$\K(V)\rtimes_\alpha G=r^*\K(V)$ with the multiplication given by
$(T,g)(S,h)=(T\alpha_g(S),gh)$ and
$(T,g)^*=(\alpha_{g}^{-1}(T^*),g^{-1})$. Our main result shows that
$\B$ and $\K(V)\rtimes_\alpha G$ are equivalent Fell bundles.

\begin{theorem}\label{thm:main}
  For $g\in G$ let $E(g)=V(r(g))\otimes_{A(r(g))} B(g)$, let
  $\E=\bigsqcup_{g\in G} E(g)$, and let $q:\E\to G$ be the projection
  map. Then $q:\E\to G$ is an upper-semicontinuous Banach bundle over
  $G$ and a $(\K(V)\rtimes_\alpha G)$--$\B$ equivalence.
\end{theorem}
\begin{proof}
  For $\xi\in \Gamma_c(G,\B)$ and $\eta\in \Gamma_c(G;\B)$, define
  $\hat{\xi}(x) := \xi|_{G_x}$ for $x\in\go$, and define a section
  $\xi\otimes\eta$ of $\E$ by
  \[
  (\xi\otimes\eta)(g)=\hat{\xi}(r(g))\otimes \eta(g).
  \]
  Then the set $\{\,\xi\otimes\eta\,:\xi,\eta\in \Gamma_c(G;\B)\}$
  satisfies the hypothesis of the Hofmann--Fell theorem. Hence there
  is a unique topology on $\E$ such that $q:\E\to G$ is an
  upper-semicontinuous Banach bundle such that the above sections are
  continuous.

  We show that $\E$ is a $(\K(V)\rtimes_\alpha G)$--$\B$
  equivalence. The right action of $\B$ on $\E$ is defined by
  \[
  (\xi\otimes a)\cdot b=\xi\otimes (ab),
  \]
  for $\xi\in V(r(g))$, $a\in B(g)$ and $b\in B(h)$, where
  $s(g)=r(h)$. It is easy to check that
  $q(\xi\otimes a\cdot b)=q(\xi\otimes a)p(b)$ and
  $\bigl(\xi\otimes a\cdot b\bigr)\cdot c=\xi\otimes a\cdot (bc)$.
  Moreover, it is a straightforward computation to show that
  $\Vert \xi\otimes a\cdot b\Vert\le \Vert \xi\otimes a\Vert \Vert
  b\Vert$
  using the $B(s(g))$-inner product on $E(g)$. Therefore the analogues
  for right actions of axioms (a)--(c) on page~40 of
  \cite{muhwil:dm08} are satisfied by the right action of $\B$ on $\E$
  (axiom~(c) contains a typographical error, and should read
  $\Vert b\cdot e\Vert \le \Vert b\Vert \cdot \Vert e\Vert$). The
  continuity of the action follows from a version of
  \cite[Proposition~C.20]{wil:crossed} for upper-semicontinuous Banach
  bundles.

  To define the left action of $\K(V)\rtimes_\alpha G$ on $\E$ we note
  that if $s(g)=r(h)$ then $V(r(h))\otimes_{A(r(h))}B(h)$ is
  isomorphic to $V(r(g))\otimes_{A(r(g))}B(gh)$. Indeed,
  $V(r(h))=V(s(g))$ is isomorphic to $V(r(g))\otimes_{A(r(g))} B(g)$
  by Lemma \ref{lem:isom_HM}, and multiplication induces an
  imprimitivity-bimodule isomorphism between
  $B(g)\otimes_{A(s(g))}B(h)$ and $B(gh)$ (see Lemma 1.2 of
  \cite{muhwil:dm08}). Moreover, $V(r(gh))=V(r(g))$.  Then for
  $(T,g)\in \K(V)\rtimes_\alpha G$,
  $\xi\otimes a_1\in V(r(g))\otimes_{A(r(g))}B(g)$, and $a_2\in B(h)$
  define
  \[
  (T,g)\cdot \beta_g(\xi\otimes a_1)\otimes a_2=(T \xi)\otimes
  (a_1a_2)\in V(r(gh))\otimes B(gh).
  \]
  Then
  $q((T,g)\cdot \beta_g(\xi\otimes a_1)\otimes
  a_2)=k(T,g)q(\beta_g(\xi\otimes a_1)\otimes a_2)=gh$.

  We must check that this left action of $\K(V)$ on $\E$ is continuous
  and satisfies axioms (a)--(c) on page~40 of
  \cite{muhwil:dm08}. Continuity follows once again from an upper
  semicontinuous Banach bundle version of
  \cite[Proposition~C.20]{wil:crossed}. Axiom~(a) is immediate from
  the definition.  We use equations \eqref{eq:7}~and~\eqref{eq:8} to
  check axiom~(b):
  \begin{align*}
    (S,t)\cdot \bigl[(T,g)\cdot\beta_{tg}(\xi\otimes b_1b_2)\otimes
    b_3\bigr]&=(S,t)\cdot \bigl[(T,g)\cdot
               \beta_g(\beta_t(\xi\otimes b_1)\otimes b_2)\otimes
               b_3\bigr]\\
             &=(S,t)\cdot T\beta_t(\xi\otimes b_1)\otimes b_2b_3\\
             &=(S,t)\cdot \beta_t(\alpha_t(T)\xi\otimes b_1)\otimes b_2b_3\\
             &=S\alpha_t(T)\xi\otimes b_1b_2b_3\\
             &=(S\alpha_t(T),tg)\cdot
               \beta_{tg}(\xi\otimes b_1b_2)\otimes b_3.
  \end{align*}
  One can easily show that
  $\Vert (T,g)\cdot \beta_g(\xi\otimes a_1)\otimes a_2\Vert \le \Vert
  T\Vert \Vert \beta_g(\xi\otimes a_1)\otimes a_2\Vert$
  using the right $A(s(g))$-inner product and the fact that $\beta_g$
  is an isometry. Therefore axiom~(c) on page 40 of \cite{muhwil:dm08}
  holds for the left action of $\K(V)$ on $\E$.

  Now we have to check that these actions of $\K(V)$ and $\B$ on $\E$
  satisfy~(a), (b)(i)--(b)(iv) and~(c) of
  \cite[Definition~6.1]{muhwil:dm08}.

  Definition~6.1(a) of \cite{muhwil:dm08} requires that the two
  actions commute, which is straightforward:
  \[
  \bigl((T,g)\cdot \beta_g(\xi\otimes a_1)\otimes a_2\bigr)\cdot
  b=(T,g)\cdot \bigl(\beta_g(\xi\otimes a_1)\otimes a_2\cdot b\bigr).
  \]

  To check (b)(i)--(b)(iv), we must first define sesquilinear maps
  $_{\K(V)\rtimes_{\alpha}G}\langle\cdot,\cdot\rangle$ from $\E *_s\E$
  to $\K(V)\rtimes_{\alpha}G$ and $\langle\cdot,\cdot\rangle_{\B}$
  from $\E *_r \E$ to $\B$. Let $g,h\in G$ such that $r(g)=r(h)$ and
  let $v\otimes a\in E(g)$ and $w\otimes b\in E(h)$. Since
  $r(g)=r(h)$, we have $v,w\in V(r(h))=V(r(g))$. Define
  \[
  \langle v\otimes a,w\otimes b\rangle_{\B}:=a^*\langle
  v,w\rangle_{A(r(h))}b\in B(g^{-1}h).
  \]
  Let $g,h\in G$ be such that $s(g)=s(h)$ and let $v\otimes a\in E(g)$
  and $w\otimes b\in E(h)$. Notice that $w\in V(r(h))=V(s(gh^{-1}))$
  and $v\in V(r(g))=V(r(gh^{-1}))$. Define
  \[
  _{\K(V)\rtimes_{\alpha}G}\langle v\otimes a,w\otimes
  b\rangle:=\bigl(v\otimes \beta_{gh^{-1}}^*(ab^*\otimes
  w^*),gh^{-1}\bigr).
  \]
  It is a routine albeit tedious task to check that these maps satisfy
  Definition~6.1(b)(i)--(b)(iv) of \cite{muhwil:dm08}; we just prove
  some of them to indicate the sorts of arguments
  involved. For~(b)(i),
  \[
  p(\langle v\otimes a,w\otimes b\rangle_{\B})=g^{-1}h
  \]
  and
  \[
  \sigma( _{\K(V)\rtimes_{\alpha}G}\langle v\otimes a,w\otimes
  b\rangle)=gh^{-1}.
  \]
  The proof of~(b)(ii) is relatively easy for the $\B$-valued
  sesquilinear form and more involved for the
  $\K(V)\rtimes_{\alpha}G$-valued sesquilinear form; we check
  both. First take $g,h\in G$ such that $r(g)=r(h)$ and let
  $v\otimes a\in E(g)$ and $w\otimes b\in E(h)$. Then
  \[
  \langle v\otimes a,w\otimes b\rangle_{\B}^* =\left(a^*\langle
    v,w\rangle_{A(r(h))}b\right)^* =b^*\langle w,v\rangle_{A(r(g))}a
  =\langle w\otimes b,v\otimes a\rangle_{\B}.
  \]
  Now fix $g,h\in G$ such that $s(g)=s(h)$ and take
  $v\otimes a\in E(g)$ and $w\otimes b\in E(h)$. Then
  \begin{align*}
    \left(_{\K(V)\rtimes_{\alpha}G}\langle v\otimes a,w\otimes
    b\rangle\right)^*
    &=\bigl(v\otimes \beta_{gh^{-1}}^*(ab^*\otimes
      w^*),gh^{-1}\bigr)^*\\
    &=\big(\alpha_{gh^{-1}}^{-1}\left(\big(v\otimes
      \beta_{gh^{-1}}^*(ab^*\otimes w^*)\big)^*\right),hg^{-1}\big)\\
    &=\big(\alpha_{gh^{-1}}^{-1}\big(\beta_{gh^{-1}}(w\otimes ba^*)\otimes
      v^*\big),hg^{-1}\big)\\
    \intertext{which, by the definition of $\alpha$, is}
    &=(\alpha_{gh^{-1}}^{-1}(\alpha_{gh^{-1}}(w\otimes
      \beta_{hg^{-1}}^*(ba^*\otimes v^*))),hg^{-1})\\
    &=(w\otimes   \beta_{hg^{-1}}^*(ba^*\otimes v^*))),hg^{-1})\\
    &=_{\K(V)\rtimes_{\alpha}G}\langle w\otimes b,v\otimes a\rangle.
  \end{align*}
  The remaining axioms (b)(iii), (b)(iv)~and~(c) of \cite[Definition
  6.1]{muhwil:dm08} are easy to prove.  Hence $\E$ is a
  $\K(V)\rtimes_\alpha G$--$\B$ equivalence.
\end{proof}

\begin{corollary}\label{cor:morequiv}
  With the notation of Theorem~\ref{thm:main}, $C^*(G;\B)$ and
  $C^*(G; \K(V)\rtimes_\alpha G)$ are Morita equivalent and so are
  $C^*_{\red}(G;\B)$ and $C^*_{\red}(G; \K(V)\rtimes_\alpha G)$.
\end{corollary}
\begin{proof}
  Theorem~\ref{thm:main} combined with Theorem~6.4 of
  \cite{muhwil:dm08} implies that $C^*(G;\B)$ and
  $C^*(G; \K(V)\rtimes_\alpha G)$ are Morita equivalent. The second
  assertion follows from Theorem~\ref{thm:main} and
  \cite[Theorem~14]{SimWil_NYJM13}.
\end{proof}

Our next corollary presents one of the many possible applications of
Theorem~\ref{thm:main} and Corollary~\ref{cor:morequiv}. The extra
hypothesis about continuity of the Fell bundle is needed in order to
cite the results of \cite{ren:jot91} which were proved in the context
of \emph{continuous} groupoid dynamical systems. Recall that if
$p:\B\to G$ is a continuous Fell bundle then $G$ acts continuously on
the primitive-ideal space of $A=\Gamma_0(\go,\B)$ with its Polish
\emph{regularized topology}\footnote{The regularized topology is
  defined in \cite[Definition~H.38]{wil:crossed} is
  Polish by \cite[Theorem~H.39]{wil:crossed}.} via equation
\eqref{eq:action} \cite[Proposition~1.14]{ren:jot91}. When we say
``$G$ acts amenably on $\Prim A$'' we require the existence of a net
of functions as in \cite[Remark~3.7]{ren:jot91}.  For this it suffices
for the Borel groupoid $\Prim A\times G$ to be measurewise amenable or
for $G$ itself to be amenable. We say that the action of $G$ on
$\Prim A$ is \emph{essentially free} if the set of points with trivial
isotropy is dense in every closed invariant set for the regularized
topology on $\Prim A$.

\begin{corollary}\label{cor:lat}
  Let $G$ be a locally compact Hausdorff groupoid and let $p:\B\to G$
  be a continuous Fell bundle. Let $A$ be the $C^*$-algebra over
  $\go$. Assume that the action of $G$ on $\Prim A$ is amenable and
  essentially free. Then the lattice of ideals of $C^*(G;\B)$ is
  isomorphic to the lattice of invariant open sets of $\Prim A$.
\end{corollary}
\begin{proof}
  Since $\V$ is a $\K(\V)$--$A$-imprimitivity bimodule, it follows
  that $\K(\V)$ and $A$ are Morita equivalent. Let
  $h:\Prim A\to\Prim \K(\V)$ be the Rieffel correspondence (see, for
  example, \cite[Corollary~3.3]{rw:morita}). Then, from the definition
  of $\V$ and $\K(\V)$ and \cite[Formula (1) on page
  1247]{ionwil:hjm11} it follows that if $P\in \Prim \K(\V)$ and
  $g\in G$ then $g\cdot P=h(g\cdot h^{-1}(P))$. Therefore the action
  of $G$ on $\Prim \K(\V)$ is amenable and essentially free.  The
  result follows from \cite[Corollary~4.9]{ren:jot91}, Remark
  \ref{rem:continuous} and Corollary~\ref{cor:morequiv}.
\end{proof}

\begin{remark}
  Corollary~\ref{cor:lat} provides an alternative proof of the
  following result, which was proved under slightly stronger
  conditions in \cite[Corollary~4.7]{KuPaSi_Sim2014}.
\end{remark}

\begin{corollary}\label{cor:sim}
  Let $G$ be a Hausdorff locally compact groupoid and let $p:\B\to G$
  be a continuous Fell bundle.  Assume that the action of $G$ on
  $\Prim A$ is amenable and essentially free. Then $C^*(G;\B)$ is
  simple if and only if the action of $G$ on $\Prim A$ is minimal.
\end{corollary}

Let $G$ be an \emph{\'{e}tale} locally compact groupoid and suppose
that the interior of the isotropy $\Iso^{\circ}(G)$ is closed.  Then
$G/\Iso^{\circ}(G)$ is also an \'{e}tale locally compact groupoid
\cite[Proposition~2.5(d)]{simwil:xx15}.

\begin{corollary}\label{cor:push-forward}
  Let $G$ be an \'{e}tale amenable locally compact groupoid, let
  $\sigma \in Z^2(G, \T)$ be a continuous 2-cocycle and suppose that
  the interior of the isotropy $\Iso^{\circ}(G)$ is closed. Suppose
  that $G/\Iso^{\circ}(G)$ is essentially principal. Then there is
  a continuous Fell bundle $p: \B \to G/\Iso^{\circ}(G)$ such that
  \[
  C^*(G/\Iso^{\circ}(G), \B) \cong C^*(G, \sigma) \quad
  \text{and}\quad C^*(\go, \B) \cong C^*(\Iso^{\circ}(G), i^*(\sigma)).
  \]
  The action of $G/\Iso^{\circ}(G)$ on
  $\Prim C^*(\Iso^{\circ}(G), i^*(\sigma))$ is essentially principal and
  the map which takes an ideal of $C^*(\Iso^{\circ}(G), i^*(\sigma))$ to
  the ideal of $C^*(G, \sigma)$ generated by its image induces an
  isomorphism from the lattice of
  $\big(G/\Iso^{\circ}(G)\big)$-invariant open sets of
  $\Prim C^*(\Iso^{\circ}(G), i^*(\sigma))$ to the lattice of ideals of
  $C^*(G, \sigma)$.
\end{corollary}

The proof of Corollary~\ref{cor:push-forward} requires a lemma.

\begin{lemma}\label{lem:continuous bundle}
Let $G$ be an \'{e}tale amenable locally compact groupoid such that $G = \Iso(G)$. Let
$\B$ be a Fell bundle over $G$ such that $\xi \mapsto \|\xi_x\|$ is continuous for $\xi
\in C^*(\go,\B)$. Then $C^*(G, \B)$ is the section algebra of a continuous field of
$C^*$-algebras over $\go$ such that $C^*(G, \B)_x \cong C^*(G_x, \B)$ for $x \in \go$.
\end{lemma}
\begin{proof}
The central inclusion of $C_0(\go)$ in $\mathcal{M} C^*(G,\B)$ is nondegenerate, so
\cite[Theorem~C.27]{wil:crossed} shows that $C^*(G,\B)$ is the section algebra of an
upper-semicontinuous field with fibres $C^*(G_x, \B)$. For lower semicontinuity, let
$\lambda$ be the faithful representation, induced by multiplication, of $C^*_r(G;\B)$ on
the Hilbert-$C^*(\go,\B)$-module completion $L^2(\B)$ of $\Gamma_c(G;\B)$ for $\langle
\xi,\eta\rangle := (\xi^*\eta)|_{\go}$ (see \cite[Proposition~3.2]{kum:pams98}). By
\cite[3.3~and~3.4]{kum:pams98}, $L^2(\B)$ is a bundle over $\go$ of Hilbert modules
$V_x$, and $\lambda$ determines representations $\lambda_x : C^*_r(G,\B) \to
\mathcal{L}(V_x)$. Since $G$ is amenable, \cite[Theorem~1]{simwil:ijm13} gives $C^*(G,\B)
= C^*_r(G,\B)$, so each $\lambda_x$ determines a faithful representation of $C^*(G,\B)_x$.
Fix $\xi \in C_c(G;\B)$, $x \in \go$, and $\varepsilon > 0$. Take $h \in L^2(\B)$ with
$\|h_x\| = \|h\| = 1$ and $\|\lambda_x(\xi)h_x\| \ge \|\xi_x\| - \varepsilon/2$. As $y
\mapsto \|(\lambda(\xi)h)_y\|$ is continuous, $\|(\lambda(\xi)h)_y)\big\| \ge
\|(\lambda(\xi)h)_x)\| - \varepsilon/2 \ge \|\xi_x\| - \varepsilon$ on some neighbourhood
$U$ of $x$. Each $\|h_y\| \le 1$, so $\|\xi_y\| \ge \|\xi_x\| - \varepsilon$ for $y \in
U$.
\end{proof}

\begin{proof}[Proof of Corollary~\ref{cor:push-forward}]
Identify $C^*(\Iso^{\circ}(G), i^*(\sigma))$ with the $C^*$-algebra of a Fell bundle $\A$
over $\Iso^{\circ}(G)$ with 1-dimensional fibres. Then $x \mapsto \|\xi_x\| = |\xi(x)|$
is continuous for $\xi \in C^*(\go, \A) = C_0(\go)$. So Lemma~\ref{lem:continuous bundle}
shows that $C^*(\Iso^{\circ}(G), i^*(\sigma))$ is the section algebra of a continuous
field of $C^*$-algebras over $\go$ with fibres $C^*(\Iso^{\circ}(G)_x, i^*(\sigma))$.
Arguing as in \cite[Proposition~4.2]{KuPaSi_Sim2014}, we obtain a continuous Fell bundle
$p: \B \to G/\Iso^{\circ}(G)$ such that
\[
  C^*(G/\Iso^{\circ}(G), \B) \cong C^*(G, \sigma) \quad\text{and}\quad
  C^*(\go; \B) \cong C^*(\Iso^{\circ}(G), i^*(\sigma)).
\]
The final assertion follows from Corollary~\ref{cor:lat}.
\end{proof}

\section{Applications to \texorpdfstring{$k$}{k}-graphs}

In this section we discuss some applications of our results to
computing the primitive-ideal spaces of twisted $k$-graph
$C^*$-algebras.

Recall (see \cite[Definition 2.1]{carkanshosim:jfa14}, and also
\cite{kumpas:nyjm00}) that if $P$ is a submonoid of an abelian group
$A$ with identity 0, then a $P$-graph is a countable small category
$\Lambda$ with a functor $d : \Lambda \to P$, called the \emph{degree
  map}, such that whenever $d(\lambda) = p+q$, there are unique
$\mu \in d^{-1}(p)$ and $\nu \in d^{-1}(q)$ such that
$\lambda = \mu\nu$ (this is called the factorisation property). We
write $\Lambda^p = d^{-1}(p)$. The factorisation property ensures that
$\Lambda^0$ is the set of identity morphisms, so we identify it with
the object set, and think of the codomain and domain maps as maps
$r, s: \Lambda \to \Lambda^0$. When $P = \N^k$, a $P$-graph $\Lambda$
is precisely a $k$-graph as introduced in \cite{kumpas:nyjm00}. As a
notational convention, given $v,w \in \Lambda^0$, we write
$v\Lambda = \{\lambda : r(\lambda) = v\}$,
$\Lambda v = \{\lambda : s(\lambda) = v\}$,
$v\Lambda w = \{\lambda : r(\lambda) = v, s(\lambda) = w\}$, and so
forth.

We say that $\Lambda$ is \emph{row-finite} if each
$|v\Lambda^p|<\infty$ is finite, and that $\Lambda$ has \emph{no
  sources} if each $v\Lambda^p \not= \emptyset$. We impose both
hypotheses throughout this section.  We recall some facts about
$P$-graphs and their groupoids from \cite[\S2]{carkanshosim:jfa14} and
\cite[\S6]{renwil:tams16}. Throughout, $P$ is a submonoid of an abelian
group $A$ as above. For more details and background, see
\cite{kumpas:nyjm00, carkanshosim:jfa14, renwil:tams16}.

Let $\le$ denote the partial order on $P$ given by $p \le q$ if and
only if $q - p \in P$. As in \cite[Example~2.2]{carkanshosim:jfa14},
there is a $P$-graph
$\Omega = \Omega_P := \{(p,q) \in P \times P : p \le q\}$ with degree
map $d(p,q) = q-p$ and range, source and composition given by
$r(p, q) = (p,p)$, $s(p,q) = (q,q)$ and $(p,q)(q,r) = (p,r)$. We have
$\Omega_P^0 = \{(p,p) : p \in P\}$ and we identify it with $P$ in the
obvious way.  If $\Lambda$ is a $P$-graph, we write $\Lambda^\Omega$
for the collection of all functors $x : \Omega_P \to \Lambda$ that
intertwine the degree maps. If $P = \N^k$, then $\Lambda^\Omega$ is
precisely the infinite-path space $\Lambda^\infty$ of
\cite[Definitions~2.1]{kumpas:nyjm00}. For $x \in \Lambda^\Omega$ we
write $x(p) := x(p, p) \in \Lambda^0$ when $p \in P$ and write
$r(x) := x(0)$.

Under the relative topology inherited from
$\prod_{(p,q) \in \Omega} \Lambda^{q-p}$, $\Lambda^\Omega$ is a
locally compact Hausdorff space with basic open sets
$Z(\lambda) = \{x \in \Lambda^\Omega : x(0, d(\lambda)) = \lambda\}$
indexed by $\lambda \in \Lambda$
\cite[page~3]{carkanshosim:jfa14}. For $x \in \Lambda^\Omega$, the
vertex $r(x) := x(0) \in \Lambda^0$ is the unique vertex such that
$x \in Z(r(x))$. More generally, $x \in Z(x(0, p))$ for every
$p \in P$. An argument like that of
\cite[Proposition~2.3]{kumpas:nyjm00} shows that if
$\lambda \in \Lambda$ and $x \in Z(s(\lambda))$, then there is a
unique element $\lambda x$ of $\Lambda^\Omega$ such that
$\lambda x \in Z(\lambda)$ and $\sigma^{d(\lambda)}(\lambda x) =
x$.
Hence, as in \cite[Remarks~2.5]{kumpas:nyjm00}, there is an action of
$P$ by local homeomorphisms on $\Lambda^\Omega$ given by
$\sigma^p(x)(q,r) = x(q+p, r+p)$. The $P$-graph $\Lambda$ is
\emph{aperiodic} if, for each $v \in \Lambda^0$ there exists
$x \in Z(v)$ such that $\sigma^p(x) \not= \sigma^q(x)$ for all
distinct $p,q \in P$. It is not hard to check that $\Lambda$ is
aperiodic if and only if the set of $x$ such that
$\sigma^p(x) \not= \sigma^q(x)$ for all distinct $p,q \in P$ is dense
in $\Lambda^\Omega$.

As in \cite[Lemma~3.1]{simwil:xx15} (see also \cite{dea,
  exeren:etds07, renwil:tams16}), the Deaconu--Renault groupoid
$G_\Lambda$ associated to the action $\sigma$ is the set
$\{(x, p-q, y) \in \Lambda^\Omega \times A \times \Lambda^\Omega :
\sigma^p(x) = \sigma^q(y)\}$,
given the topology generated by the sets
$Z(\mu,\nu) = \{(\mu x, d(\mu) - d(\nu), \nu x) : x \in Z(s(\mu))\}$
indexed by pairs $\mu,\nu \in \Lambda$ with $s(\mu) = s(\nu)$. The
unit space is $\{(x, 0, x) : x \in \Lambda^\Omega\}$, which we
identify with $\Lambda^\Omega$ (the topologies agree), and the
structure maps are $r(x,g,y) = x$, $s(x,g,y) = y$,
$(x,g,y)^{-1} = (y, -g, x)$ and $(x, g, y)(y,h,z) = (x, g+h,z)$. This
groupoid is second countable, \'etale and amenable
\cite[Proposition~5.12 and Theorem~5.13]{renwil:tams16}, and the basic
open sets described above are compact open bisections
\cite[page~4]{carkanshosim:jfa14}. For more on Deaconu--Renault
groupoids of the sort we study here, see \cite{dea, exeren:etds07,
  renwil:tams16}.

We denote the $G_\Lambda$-orbit $\{\lambda\sigma^p(x) : p \in P, \lambda \in \Lambda
x(p)\}$ of $x \in \Lambda^\omega$ by $[x]$.

We shall restrict attention to abelian monoids $P$ which arise as the
image of $\N^k$ in $\Z^k/H$ for some subgroup $H$ of $\Z^k$.  We begin
by presenting a characterisation of the $P$-graphs $\Lambda$ such that
$G_\Lambda$ is essentially principal. Before doing that, we need to
introduce an order relation on $\Lambda^\Omega$.

\begin{definition}
  Let $H$ be a subgroup of $\Z^k$ and let $P$ be the image of $\N^k$
  in $\Z^k/H$. Let $\Lambda$ be a row-finite $P$-graph with no
  sources. Given $x,y \in \Lambda^\Omega$, we write $x \preceq y$ if
  for every $m \in P$ there exists $n \in P$ such that
  $x(m) \Lambda y(n) \not= \emptyset$. We write $x \sim y$ if
  $x \preceq y$ and $y \preceq x$.
\end{definition}

\begin{lemma}\label{lem:orbit containment}
  Let $H$ be a subgroup of $\Z^k$ and let $P$ be the image of $\N^k$
  in $\Z^k/H$. Let $\Lambda$ be a row-finite $k$-graph with no
  sources. For $x,y \in \Lambda^\Omega$, we have $x \preceq y$ if and
  only if $\overline{[x]} \subseteq \overline{[y]}$. In particular,
  $x \sim y$ if and only if $\overline{[x]} = \overline{[y]}$.
\end{lemma}
\begin{proof}
  First suppose that $x \preceq y$. We must show that
  $\overline{[x]} \subseteq \overline{[y]}$. Since $\overline{[y]}$ is
  invariant and closed, it suffices to show that
  $x \in \overline{[y]}$; that is, that every neighbourhood of $x$
  intersects $[y]$.  Fix a basic open neighbourhood $Z(x(0,m))$ of
  $x$. Since $x \preceq y$, there exists $n \in P$ such that
  $x(m)\Lambda y(n) \not= \emptyset$, say
  $\lambda \in x(m)\Lambda y(n)$. Then
  $x(0,m)\lambda\sigma^n(y) \in Z(x(0,m)) \cap [y]$.

  Now suppose that $\overline{[x]} \subseteq \overline{[y]}$. Then in
  particular $x \in \overline{[y]}$, so for fixed $m \in P$, the set
  $[y]$ meets the basic open neighbourhood $Z(x(0,m))$ of $x$, say at
  $z = x(0,m)z'$. By definition, we have $\sigma^p(z) = \sigma^q(y)$
  for some $p,q \in P$. Choose $r \in P$ such that $r-p, r-m \in P$,
  and let $n = r - p + q$. Then
  $\sigma^n(y) = \sigma^{r-p}(\sigma^q(y)) = \sigma^{r-p}(\sigma^p(z))
  = \sigma^r(z)$.
  So $z'(0, r-m) \in x(m)\Lambda y(n)$. Hence $x \preceq y$.
\end{proof}

The next lemma characterises when a $P$-graph groupoid is essentially
principal in terms of the order structure just discussed. Following
the standard definition for $k$-graphs \cite{raesimyee:pems03} (see
also \cite[]{renwil:tams16}), we say that a subset
$H \subseteq \Lambda^0$ of the vertex set of a row-finite $P$-graph
$\Lambda$ with no sources is \emph{hereditary} if
$H\Lambda \subseteq \Lambda H$ and \emph{saturated} if whenever
$v\Lambda^p \subseteq \Lambda H$, we have $v \in H$. A subset $T$ of
$\Lambda^0$ is a \emph{maximal tail} if its complement is a saturated
hereditary set and $s(v\Lambda) \cap s(w\Lambda) \not= \emptyset$ for
all $v,w \in T$.  We say that $\Lambda$ is \emph{strongly aperiodic}
if for every saturated hereditary subset $H \subseteq \Lambda^0$, the
subgraph $\Lambda \setminus \Lambda H$ is aperiodic (see
\cite{kanpas:ijm14}).

\begin{lemma}\label{lem:strong aperiodicity}
  Let $H$ be a subgroup of $\Z^k$ and let $P$ be the image of $\N^k$
  in $\Z^k/H$. Let $\Lambda$ be a row-finite $P$-graph with no
  sources, and let $G_\Lambda$ be the associated groupoid. Then the
  following are equivalent.
  \begin{enumerate}
  \item\label{it:ess-prin} The groupoid $G_\Lambda$ is essentially
    principal in the sense that the points with trivial isotropy are
    dense in every closed invariant subspace of $G_\Lambda^{(0)}$.
  \item\label{it:str-ap} The $P$-graph $\Lambda$ is strongly
    aperiodic.
  \item\label{it:ap-tails} The subgraph $\Lambda T$ is aperiodic for
    every maximal tail $T$ of $\Lambda$.
  \item\label{it:equiv-ap} For every $y \in \Lambda^\Omega$, there is
    an aperiodic path $x \in \Lambda^\Omega$ such that $x \sim y$.
  \end{enumerate}
\end{lemma}
\begin{proof}
  (\ref{it:ess-prin})${\implies}$(\ref{it:str-ap}) Suppose that
  $G_\Lambda$ is essentially principal, and fix a saturated hereditary
  $H \subseteq \Lambda^0$. Then
  $(\Lambda \setminus \Lambda H)^\Omega \subseteq \Lambda^\Omega$ is a
  closed invariant set, and hence
  $G_\Lambda|_{(\Lambda \setminus \Lambda H)^\Omega}$ is topologically
  free. The argument of \cite[Proposition~4.5]{kumpas:nyjm00} then
  shows that $\Lambda \setminus \Lambda H$ is aperiodic.

  (\ref{it:str-ap})${\implies}$(\ref{it:ap-tails}) If
  $\Lambda \setminus \Lambda H$ is aperiodic for every saturated
  hereditary $H$, then in particular, every $\Lambda T$ is aperiodic
  because the complement of a maximal tail is saturated and
  hereditary.

  (\ref{it:ap-tails})${\implies}$(\ref{it:equiv-ap}) Suppose that
  $\Lambda T$ is aperiodic for every maximal tail $T$. For
  $y \in \Lambda^\Omega$, the set
  $T_y := \{z(n) : n \in P, z \in [y]\}$ is a maximal tail, and we
  have $\overline{[y]} = (\Lambda T_y)^\Omega$. List
  $P \times P \setminus \{(n,n) : n \in P\}$ as
  $(m_i, n_i)^\infty_{i=1}$. Let $\mathbf{1} \in P$ denote the image
  of $(1,\dots, 1) \in \N^k$ under the quotient map from $\Z^k$ to
  $\Z^k/H$. We claim that there is a sequence
  $(\mu_i, q_i)^\infty_{i=0} \in r(y)\Lambda T \times P$ with the
  following properties:
  \begin{itemize}
  \item $d(\mu_i) \ge i\cdot\mathbf{1}$ for all $i \ge 0$;
  \item $\mu_{i+1} \in \mu_i \Lambda y(q_{i+1})$ for all $i \ge 0$;
    and
  \item for each $i \ge 1$ and each $1 \le j \le i$ there exists $l$
    such that $\mu_i(m_j, m_j+l) \not= \mu_i(n_j, n_j + l)$.
  \end{itemize}
  Set $\mu_0 = r(y)$ and $q_0 = 0$; this trivially has the desired
  properties. We construct the $\mu_i$ inductively. Given $\mu_i$ and
  $q_i$, note that $\mu_i \sigma^{q_i}(y) \in (\Lambda T_y)^\Omega$.
  Since $\Lambda T_y$ is aperiodic there is an aperiodic infinite path
  $x_{i+1}$ in $Z(\mu_i) \cap (\Lambda T)^\Omega$. Since $x_{i+1}$ is
  aperiodic, we can choose $l \in P$ such that
  $x_{i+1}(m_{i+1}, m_{i+1} + l) \not= x_{i+1}(n_{i+1}, n_{i+1} +
  l)$. Choose $p \in P$ such that
  \[
  p \ge d(\mu_i) + \mathbf{1},\quad p \ge (m_{i+1} + l)\quad\text{ and
  }\quad p \ge n_{i+1} + l.
  \]
  Since $x_{i+1} \in (\Lambda T)^\Omega$, there exists
  $q_{i+1} \ge q_i + \mathbf{1}$ such that
  $x_{i+1}(p) \Lambda y(q_{i+1}) \not= \emptyset$. Now this choice of
  $q_{i+1}$ and any choice of
  $\mu_{i+1} \in x_{i+1}(0,p)\Lambda y(q_{i+1})$ satisfies the three
  bullet points, completing the proof of the claim.

  Let $x \in \Lambda^\Omega$ be the unique element such that
  $x(0, d(\mu_i)) = \mu_i$ for all $i$. By construction of the $\mu_i$
  we have $\sigma^{m_i}(x) \not= \sigma^{n_i}(x)$ for all $i$, and so
  $x$ is aperiodic. For each $m \in P$ we can choose $i$ such that
  $d(\mu_i) \ge m$ and $q_i \ge m$. The first condition forces
  $x(m) \Lambda y(q_i) \not= \emptyset$, so that $x \preceq y$; and
  the second condition forces
  $y(m)\Lambda x(d(\mu_{i+1})) \not=\emptyset$, so that $y \preceq
  x$. Hence $y \sim x$ as required.

  (\ref{it:equiv-ap})${\implies}$(\ref{it:ess-prin}) Fix a closed
  invariant $X \subseteq G_\Lambda^{(0)}$ and $y \in X$.
  By~(\ref{it:equiv-ap}), there is an aperiodic infinite path $x$ such
  that $x \sim y$. Lemma~\ref{lem:orbit containment} gives
  $\overline{[x]} = \overline{[y]} \subseteq X$, so there is a
  sequence $(y_n)$ in $[x]$ converging to $y$. Each $y_n$ is a point with trivial
  isotropy because $x$ is aperiodic. So $G_\Lambda$ is essentially principal.
\end{proof}

For the definition of the $C^*$-algebras of the $P$-graphs considered here, see
\cite[Section~2]{carkanshosim:jfa14}; for $k$-graphs, the definition appeared first in
\cite{kumpas:nyjm00}. For our purposes, it suffices to recall first that $C^*(\Lambda)$
is isomorphic to $C^*(G_\Lambda)$ \cite[Proposition~2.7]{carkanshosim:jfa14}, and second
that this isomorphism intertwines the \emph{gauge action} of $(\Z^k/H)\widehat{\;} \cong
H^\perp \subseteq \T^k$ on $C^*(\Lambda)$ with the action of $(\Z^k/H)\widehat{\;}\,$ on
$C^*(G_\Lambda)$ determined by $(\chi \cdot f)(x, g, y) = \chi(g) f(x,g,y)$ for $f \in
C_c(G_\Lambda)$, $\chi \in (\Z^k/H)\widehat{\;}$, and $(x,g,y) \in G_\Lambda$. An ideal
of $C^*(\Lambda)$ is \emph{gauge-invariant} if it is invariant for this gauge action.

In the situation where $H$ is trivial in the preceding lemma so that
the statement is about $k$-graphs, we could use this result combined
with Corollary~\ref{cor:lat} to describe the primitive-ideal spaces of
the $C^*$-algebras of strongly-aperiodic $k$-graphs. However, this
result already follows from Renault's results about groupoid dynamical
systems:

\begin{corollary}
  Let $H$ be a subgroup of $\Z^k$ and let $P$ be the image of $\N^k$
  in $\Z^k/H$. Let $\Lambda$ be a row-finite $P$-graph with no
  sources, and let $G_\Lambda$ be the associated groupoid. Then the
  following are equivalent.
  \begin{enumerate}
  \item The $P$-graph $\Lambda$ is strongly aperiodic.
  \item The map
    $I \mapsto (I \cap C_0(G_\Lambda^{(0)}))\widehat{\;}\;$ is a
    lattice isomorphism between the lattice of ideals of
    $C^*(G_\Lambda)$ and the lattice of open invariant subsets of
    $G_\Lambda^{(0)}$.
  \item Every ideal of $C^*(\Lambda)$ is gauge-invariant.
  \end{enumerate}
\end{corollary}
\begin{proof}
  Lemma~\ref{lem:strong aperiodicity} shows that $\Lambda$ is strongly
  aperiodic if and only if $G_\Lambda$ is essentially principal. Since
  $G_\Lambda$ is amenable \cite[Lemma~3.5]{SimWil_NYJM13},
  (1)${\implies}$(2) therefore follows from
  \cite[Corollary~4.9]{ren:jot91}.  Since $C_0(G_\Lambda^{(0)})$ is
  pointwise fixed by the gauge action on $C^*(\Lambda)$, we have
  (2)${\implies}$(3). For (3)$\implies$(1), we argue the
  contrapositive. Suppose that $\Lambda$ is not strongly aperiodic,
  and so $G_\Lambda$ is not essentially principal. So there exists
  $x \in \Lambda^\Omega$ such that $G_\Lambda|_{\overline{[x]}}$ is
  not topologically principal. By \cite[Lemma~3.1]{broclafarsim:sf14}
  there is an open bisection $U$ that is interior to the isotropy of
  $G_\Lambda|_{\overline{[x]}}$ but contains no units. By definition
  of the topology on $G_\Lambda$, we can assume that $U$ is clopen and
  has the form $U = \{(y, p, y) : y \in K\}$ for some compact
  relatively open $K \subseteq \overline{[x]}$ and some
  $p \in \Z^k/H \setminus \{0 + H\}$.  Fix a character $\chi$ of
  $\Z^k/H$ such that $\chi(p) \not= 1$. Choose any
  $a \in C_c(G_\Lambda)$ whose restriction to
  $G_\Lambda|_{\overline{[x]}}$ is equal to $1_U - \chi(p)1_K$. As in
  the proof of \cite[Proposition~5.5]{carkanshosim:jfa14}, let $\pi_x$
  be the representation of $C^*(G_\Lambda)$ on $\ell^2([x])$ given by
  $\pi_x(f)\delta_y = \sum_{g \in (G_\Lambda)_y} f(g)\delta_{r(g)}$.
  Then $\pi_x(a) \delta_x = (1 - \chi(p))\delta_x \not= 0$, and
  $\pi_x(\chi\cdot a) = \pi_x(\chi(p)1_U - \chi(p)1_K) =
  \chi(p)\pi_x(1_U - 1_K) = 0$.
  So $\ker(\pi_x)$ is not gauge-invariant.
\end{proof}

If $\Lambda$ is a $P$-graph and $q : \N^k \to P$ is a homomorphism,
then the pullback $k$-graph $q^*\Lambda$ is given by
$q^*\Lambda = \{(\lambda, m) \in \Lambda \times \N^k : d(\lambda) =
q(m)\}$
with pointwise operations and degree map $d(\lambda,m) = m$ (see
\cite[Definition~3.1]{carkanshosim:jfa14} or
\cite[Definition~1.9]{kumpas:nyjm00}).

Recall from \cite[Definition 3.5]{kumpassim:tams15} that a $\T$-valued
2-cocycle on a $k$-graph $\Lambda$ is a map
$c : \{(\mu,\nu) \in \Lambda \times \Lambda : s(\mu) = r(\nu)\} \to
\T$
such that $c(r(\lambda),\lambda) = 1 = c(\lambda, s(\lambda))$ for all
$\lambda$ and such that
\[
c(\lambda,\mu)c(\lambda\mu,\nu) = c(\mu,\nu)c(\lambda,\mu\nu)\text{
  for all composable $\lambda,\mu,\nu$.}
\]
Again, rather than present a definition of the twisted $C^*$-algebra
$C^*(\Lambda, c)$, we just recall from
\cite[Corollary~7.8]{kumpassim:tams15} that for each 2-cocycle $c$ on
$\Lambda$ there is a locally constant 2-cocycle $\sigma$ on
$G_\Lambda$ such that $C^*(\Lambda, c) \cong C^*(G_\Lambda, \sigma)$.

Our next result, which provides a method for computing the
primitive-ideal space of a twisted $C^*$-algebra associated to such a
$k$-graph obtained as a pullback of a strongly aperiodic $P$-graph,
follows easily from Corollary~\ref{cor:push-forward} and
Lemma~\ref{lem:pullback groupoid}. We provide a description of the
topology on $\Prim C^*(\Iso^{\circ}(G_\Lambda), \sigma)$ that will
help in applying the theorem in Proposition~\ref{prp:primtop} below.

Recall that if $\sigma$ is a $2$-cocycle on an abelian group $H$, then
the symmetry group or symmetrizer subgroup of $\sigma$ is
\[
S_{\sigma}=\set{t\in H: \text{$\sigma(t,s)=\sigma(s,t)$ for all
    $s\in H$}}.
\]
Note that $S_{\sigma}$ is also the kernel $Z(h_{\sigma})$ of the map
$h_{\sigma}:H\to \hat H$ given by
$h_{\sigma}(s)(t)=\sigma(s,t)\overline{\sigma(t,s)}$.  We can view
$h_{\sigma}$ as an antisymmetric bicharacter on $H$. The map
$\sigma\mapsto \sigma\sigma^{*}$, where
$\sigma^{*}(s,t)=\overline{\sigma(t,s)}$, is a isomorphism of
$H^{2}(H,\mathbb T)$ with the group $X(H,\mathbb T)$ of antisymmetric
bicharacters on $H$ (see \cite[Proposition~3.2]{olepedtak:jot80}).  It
is well-known (see \cite{bagkle:jfa73} or
\cite[Proposition~34]{gre:am78}) that the primitive-ideal space of the
twisted group $C^{*}$-algebra $C^{*}(H,\sigma)$ is homeomorphic to the
dual of $S_{\sigma}$.

\begin{theorem}\label{thm:pullback}
  Let $H$ be a subgroup of $\Z^k$ and let $P$ be the image of $\N^k$
  in $\Z^k/H$. Let $\Gamma$ be a row-finite $P$-graph with no sources,
  and let $\Lambda := q^*\Gamma$ be the pullback $k$-graph along the
  quotient map $q : \Z^k \to \Z^k/H$. Suppose that
  $\Gamma$ is strongly aperiodic, let $c$ be a $\T$-valued 2-cocycle
  on $\Lambda$ and let $\sigma$ be a locally constant $\T$-valued
  2-cocycle on $G_\Lambda$ such that
  $C^*(\Lambda, c) \cong C^*(G_\Lambda, \sigma)$. Then
  $\Iso^{\circ}(G_\Lambda) \cong \Lambda^\infty \times H$ and is
  closed in $G_\Lambda$. For $x \in \Lambda^\infty$, let
  $\sigma_x$ be the restriction of $\sigma$ to
  $\{(x, h, x) : h \in H\} \subseteq (G_\Lambda)^x_x$. Then
  $C^*(\Iso^{\circ}(G_\Lambda), \sigma)$ is the section algebra of a
  continuous field of $C^*$-algebras such that
  $C^*(\Iso^{\circ}(G_\Lambda), \sigma)_x \cong C^*(H, \sigma_x)$ with
  $\Prim C^*(H, \sigma_x) = \hat S_{\sigma_{x}}$. There is an action
  of $G_\Gamma$ on $\Prim C^*(\Iso^{\circ}(G_\Lambda), \sigma)$ such
  that
  \[
  \text{$\Prim C^*(\Lambda, c)$ is homeomorphic to $  \Prim
    C^*(\Iso^{\circ}(G_\Lambda),
  \sigma)/G_\Gamma$.}
  \]
\end{theorem}

To prove the theorem, we need the following lemma about the structure
of the groupoid of a pullback.  Let $G$ be an \'{e}tale groupoid, let
$t: A \to B$ be a homomorphism of locally compact abelian groups such
that $\ker t$ is discrete, and let $c: G \to B$ be a continuous
1-cocycle. Then the pullback
\[
t^*(G) = \{ (g, a) \in G \times A : c(g) = t(a) \}
\]
is a closed subgroupoid of $G \times A$ and is \'{e}tale in the
relative topology. The projection map $\pi_1 : t^*(G) \to G$ onto the
first coordinate is a groupoid homomorphism. If $t$ is surjective,
then $\pi_1$ is a (surjective) local homeomorphism
(see~\cite{ionkum16}).

\begin{lemma}\label{lem:pullback groupoid}
  Let $H$ be a subgroup of $\Z^k$ and let $P$ be the image of $\N^k$
  in $\Z^k/H$. Let $\Gamma$ be a row-finite $P$-graph with no sources,
  and let $\Lambda := q^*\Gamma$ be the pullback $k$-graph with
  respect to the quotient map $q : \Z^k \to \Z^k/H$. There is a
  homeomorphism $q^* : \Gamma^\Omega \to \Lambda^\infty$ such that
  $q^*(x)(m,n) = (x(q(m), q(n)), n-m)$ for all $x \in \Gamma^\Omega$
  and $m \le n \in \N^k$. For all $m \in \N^k$ and
  $x \in \Gamma^\Omega$ we have
  $\sigma^m(q^*(x)) = q^*(\sigma^{q(m)}(x))$.

  There is a groupoid isomorphism
  $\tilde{q} : q^*(G_\Gamma) \to G_\Lambda$ such that
  \begin{equation}\label{eq:pullbackiso}
    \tilde{q} (((x, p, y), m)) = (q^*(x), m, q^*(y))
  \end{equation}
  for all $(x, p, y) \in G_\Gamma$, $m \in \Z^k$ such that $q(p) = m$.
  The map $\pi_\infty = \pi_1 \circ \tilde{q}^{-1}$ defines a groupoid
  homomorphism $\pi_\infty : G_\Lambda \to G_\Gamma$ such that
  \[
  \pi_\infty(q^*(x), m, q^*(y)) = (x, q(m), y) \quad\text{ for
  }(q^*(x), m, q^*(y)) \in G_\Lambda,
  \]
  and the restriction of $\pi_\infty$ to
  $G_\Lambda^{(0)} \cong \Lambda^\infty$ is $(q^*)^{-1}$.
\end{lemma}
\begin{proof}
  The map $q^*$ is injective by construction. To see that it is
  surjective, fix $y \in \Lambda^\infty$.  We prove that there is a
  well-defined map $\tilde{\pi}(y) : \Omega_P \to \Gamma$ such that
  $\tilde{\pi}(y)(q(m), q(n)) = \pi(y(m,n))$ for all
  $(m, n) \in \Omega_k$ and $q^*(\tilde{\pi}(y)) = y$.  Fix $m \le n$
  and $m' \le n'$ in $\N^k$ such that $q(m) = q(m')$ and
  $q(n) = q(n')$.  Let $\pi : \Lambda \to \Gamma$ be the quotient map
  $\pi(\alpha, m) = \alpha$. We have $y(m,n) = (\pi(y(m,n)), n-m)$ and
  $y(m',n') = (\pi(y(m',n')), n-m)$. We claim that
  $\pi(y(m,n)) = \pi(y(m',n'))$. To see this, let $p := n\vee
  n'$. Since $\pi$ is a functor,
  \[
  \pi(y(0, m))\pi(y(m,n))\pi(y(n,p)) = \pi(y(0,p)) = \pi(y(0,
  m'))\pi(y(m',n'))\pi(y(n',p)).
  \]
  We have $d(\pi(y(0, m))) = q(m) = q(m') = d(\pi(y(0, m')))$, and the
  same calculation gives $d(\pi(y(m, n))) = d(\pi(y(m', n')))$ and
  $d(\pi(y(n, p))) = d(\pi(y(n', p)))$. So the factorisation property
  in $\Gamma$ guarantees that $\pi(y(m,n)) = \pi(y(m',n'))$. It
  follows that there is a well-defined map
  $\tilde{\pi}(y) : \Omega_P \to \Gamma$ such that
  $\tilde{\pi}(y)(q(m), q(n)) = \pi(y(m,n))$ for all
  $(m, n) \in \Omega_k$. It is routine to check that
  $q \times q : (m,n) \mapsto (q(m), q(n)$ is a surjection from
  $\Omega_k$ to $\Omega_P$, that each
  $\tilde{\pi}(y) \in \Gamma^\Omega$ and that each
  $q^*(\tilde{\pi}(y)) = y$.  So $q^*$ is surjective.

  Since $(q^*)^{-1}(Z(\lambda)) = Z(\pi(\lambda))$ for
  $\lambda \in \Lambda$, and
  $q^*(Z(\gamma)) = \bigcup_{q(m) = d(\gamma)} Z((\gamma, m))$ for
  each $\gamma \in \Gamma$, we see that $q^*$ is continuous and
  open. Moreover, for every $\gamma \in \Gamma$, $m \in \N^k$ and
  $x \in \Gamma^\Omega$ such that $d(\gamma) = q(m)$ and
  $s(\gamma) = r(x)$ we have $q^*(\gamma x) = (\gamma, m)q^*(x)$.
  Hence, for all $m \in \N^k$ and $x \in \Gamma^\Omega$ we have
  $\sigma^m(q^*(x)) = q^*(\sigma^{q(m)}(x))$.

  We must next show that equation~\eqref{eq:pullbackiso} gives a
  well-defined map $\tilde{q} : q^*(G_\Gamma) \to G_\Lambda$.  Given
  $(x, p, y) \in G_\Gamma$, take $m \in \Z^k$ such that $q(m) = p$.
  We have $p = i - j$ for some $i, j \in P$ such that
  $\sigma^{i}(x) = \sigma^{j}(y)$.  By definition of $P := q(\N^k)$,
  we can write $i = q(a_0)$ and $j = q(b_0)$ for some
  $a_0, b_0 \in \N^k$.  Now $c := m - (a_0 - b_0) \in \ker(q) = H$,
  and so we have $c = c^+ - c^-$ for some $c^+, c^- \in \N^k$.  Since
  $q(c) = 0$, we have $q(c^+) = q(c^-)$.  %% define $l := q(c^+)$.
  Now setting $a := a_0 + c^+$ and $b := b_0 + c^-$ we have
  $a, b \in \N^k$ and $m = a-b$, $q(a) = i + l$ and $q(b) = j + l$
  where $l = q(c^+)$.  Hence,
  \[
  \sigma^a(q^*(x)) = q^*(\sigma^{q(a)}(x)) = q^*(\sigma^{i+l}(x)) =
  q^*(\sigma^{j+l}(x)) = \sigma^b(q^*(x)),
  \]
  since $\sigma^{i}(x) = \sigma^{j}(y)$.  Therefore
  $(q^*(x), m, q^*(y)) \in G_\Lambda$ and $\tilde{q}$ is well defined.
  By construction $\tilde{q}$ is a coninuous groupoid homomorphism
  with inverse given by
  \[
  G_\Lambda \owns (z, m, w) \mapsto (((q^*)^{-1}(z), q(m),
  (q^*)^{-1}(w)), m) \in q^*(G_\Gamma).
  \]
  Since $\pi_\infty$ is a composition of groupoid homomorphisms it is
  also a groupoid homomorphism.  The remaining assertions are
  straightforward.
\end{proof}

\begin{proof}[Proof of Theorem~\ref{thm:pullback}]
  To deduce this theorem from Corollary~\ref{cor:push-forward}, we
  just need to establish that
  $\Iso^{\circ}(G_\Lambda) \cong \Lambda^\infty \times H$ and is
  closed in $G_\Lambda$. Let $\pi_\infty : G_\Lambda \to G_\Gamma$ be
  the groupoid homomorphism of Lemma~\ref{lem:pullback
    groupoid}. Since $G_\Gamma$ is essentially principal, we have
  $\Iso^{\circ}(G_\Lambda) = \pi_{\infty}^{-1}(G_\Gamma^{(0)})$, which
  is clopen---and in particular closed---because $\pi_\infty$ is
  continuous. The definition of $\pi_\infty$ shows that
  $\pi_\infty^{-1}(G_\Gamma^{(0)}) = \{(x, m, x) : x \in
  \Lambda^\infty, m \in H\}$.
  By Lemma~\ref{lem:pullback groupoid},
  $\pi_\infty|_{\{(x,m,x) : x \in \Lambda^\infty\}}$ is a
  homeomorphism onto $\Gamma^\Omega$ for each $m$, and so
  $\Iso^{\circ}(G_\Lambda) \cong \Lambda^\infty \times H$ as
  topological spaces.
\end{proof}

\begin{proposition}\label{prp:primtop}
  Resume the notation and hypotheses of
  Theorem~\ref{thm:pullback}. For each antisymmetric bicharacter
  $\omega$ of $H$, the set
  $C_\omega := \{x \in \Lambda^\infty : \sigma_x \sigma^*_x =
  \omega\}$
  is a clopen invariant subset of $\Lambda^\infty$. The set $\Xi$ of
  antisymmetric bicharacters such that $C_\omega \not= \emptyset$ is
  countable, and $\Prim C^*(\Iso^{\circ}(G_\Lambda), \sigma)$ is
  homeomorphic to the topological disjoint union
  $\bigsqcup_{\omega \in \Xi} C_\omega \times
  {Z(\omega)}\widehat{\;}$.
\end{proposition}
\begin{proof}
  For each $x \in \Lambda^\infty$, write
  $\omega_x := \sigma_x \sigma^*_x$. As in the first part of the proof
  of \cite[Lemma~3.3]{KuPaSi_Sim2014}, the map $x \mapsto \omega_x$ is
  locally constant because $\sigma$ is locally constant and $H$ is
  finitely generated.  The second part of the proof shows that the
  cohomology class of $\sigma_x$ is constant along orbits; since
  $\sigma \mapsto \sigma\sigma^*$ induces an isomorphism of
  $H^2(H, \T)$ with the group of antisymmetric bicharacters (see
  \cite[Proposition~3.2]{olepedtak:jot80}), it follows that
  $x \mapsto \omega_x$ is constant along orbits as well.  Since every
  locally constant function is continuous, it follows that
  $x \mapsto \omega_x$ is constant on orbit closures.

  Since $x \mapsto \omega_x$ is locally constant, for each bicharacter
  $\omega$, the set
  $C_\omega = \{y \in \Lambda^\infty : \omega_y = \omega\}$ is
  clopen. This $C_\omega$ is also invariant because
  $x \mapsto \omega_x$ is constant along orbits. Choose an increasing
  sequence $F_n$ of finite subsets of $\Lambda^0$ such that
  $\bigcup F_n = \Lambda^0$. Each
  $K_n := \bigcup_{v \in F_n} Z(v) \subseteq \Lambda^\infty$ is
  compact, so for each $n$ the set $K_n$ is covered by finitely many
  $C_\omega$. Since the $C_\omega$ are mutually disjoint, it follows
  that $\Xi_n := \{\omega : C_\omega \cap K_n \not= \emptyset\}$ is
  finite. So $\Xi = \bigcup_n \Xi_n$ is countable.

  We now have
  $C^*(\Iso^{\circ}(G_\Lambda), \sigma) = \bigoplus_{\omega \in \Xi}
  C^*(\Iso^{\circ}(G_\Lambda)|_{C_\omega}, \sigma)$.
  Fix $\omega \in \Xi$, and define
  $\mathcal{I}_\omega := \Iso^{\circ}(G_\Lambda)|_{C_\omega} \cong
  C_\omega \times H$.
  The argument of the second and third paragraphs in the proof of
  \cite[Proposition~3.1]{KuPaSi_Sim2014} shows that there is a
  2-cocycle $\sigma_\omega$ of $H$ and a 1-cochain $b$ of
  $\mathcal{I}_\omega$ such that
  $(\delta^1 b) \sigma|_{\mathcal{I}_\omega} = 1_{C_\omega} \times
  \sigma_\omega$.
  Hence
  $C^*(\mathcal{I}_\omega, \sigma|_{\mathcal{I}_\omega}) \cong
  C_0(C_\omega) \otimes C^*(H, \sigma_\omega)$
  by \cite[Proposition~II.1.2]{ren:groupoid}.

  As observed earlier,
  $\Prim C^*(H, \sigma_\omega) \cong \hat
  S_{\sigma_\omega\sigma^*_\omega} = Z(\omega)\widehat{\;}$;
  the result follows.
\end{proof}

% \bibliography{references}
\printbibliography

\end{document}